\documentclass[11pt,oneside]{amsart}
\usepackage[english]{babel}
\usepackage{fullpage}
\usepackage{amssymb,pstricks,amscd,epsfig}
\usepackage{graphicx}

\usepackage{amsmath, amsthm, amssymb}
\usepackage{pdfsync}
\usepackage[latin1]{inputenc}
\usepackage{stmaryrd}
\usepackage{color}
\usepackage{setspace}
\usepackage[matrix,arrow,tips,curve]{xy}

\usepackage{hyperref}
\hypersetup{backref=true}

\usepackage{cite}
\newcommand{\Q}{\mathbb{Q}}

\newcommand{\C}{\mathbb{C}}
\renewcommand{\P}{\mathbb{P}}

\DeclareMathOperator{\Stab}{Stab}
\DeclareMathOperator{\Hilb}{Hilb}
\newcommand{\mc}{\mathcal}

\newcommand{\ov}{\overline}
\newcommand{\be}{\begin{equation}}
\newcommand{\ee}{\end{equation}}

\newcommand{\wt}{\widetilde}

\newcommand{\hk}{hyper--K\"ahler }

\newtheorem{thm}{Theorem}
\newtheorem{thmintro}{Theorem}{\Alph{thmintro}}
\newtheorem{prop}[thm]{Proposition}
\newtheorem{lemma}[thm]{Lemma}
\theoremstyle{definition}

\newtheorem{rem}[thm]{Remark}
\newtheorem{defin}[thm]{Definition}
\newtheorem{cor}[thm]{Corollary}

\newtheorem{example}[thm]{Example}
\newtheorem{question}[thm]{Question}

\numberwithin{thm}{section}

\numberwithin{equation}{section}
\DeclareMathOperator{\Sym}{Sym}

\DeclareMathOperator{\Pic}{Pic}

\DeclareMathOperator{\im}{Im}

\DeclareMathOperator{\Aut}{Aut}

\DeclareMathOperator{\Sing}{Sing}

\DeclareMathOperator{\Def}{Def}

\DeclareMathOperator{\CH}{CH}

\author[G. Sacc\`a]{Giulia Sacc\`a}
\address{Mathematics Department\\
Columbia University \\
Mathematics Department\\
2990 Broadway\\
New York, NY 10027}

\title{Compactifying Lagrangian fibrations}

\begin{document}

\begin{abstract}
We suggest a general framework for compactifing quasi-projective Lagrangian fibrations of geometric origin by holomorphic symplectic varieties. This framework includes a compactification criterion, which we then apply to various fibrations of geometric origin, and a discussion on holomorphic forms that are defined via correspondences in geometric examples. As application, we show that given a   Lagrangian fibration $X \to B$ admitting local sections over an open subset $V$ with codimension $\ge 2$ complement, there exists a (possibly singular) holomorphic symplectic compactification of the Albanese fibration $A \to V$ (which we show exists as a smooth commutative algebraic group with connected fibers acting on $X_{V}$), as well as of any other torsor over $A$, or over any smooth commutative group scheme over $B$ with connected fibers that is isogenous to $A$.

%there exist holomorphic symplectic compactification of the Albanese fibration $A \to B$ of a Lagrangian fibration $X \to B$ admitting local sections (we show this is a smooth commutative algebraic group with connected fibers, establishing a result of independent interest), and of any other fibration that can be obtained as torsors over $A$ or over smooth commutative group schemes that are isogenous to $A$.
\end{abstract}

\maketitle

\section{Introduction}

The aim of this paper is to give a general framework to address the following question:

\begin{question} Let $U \subseteq B$ be an open subset of a projective variety which is the base of a Lagrangian fibration $\pi_U : M_U \to U$. When does there exist a \hk compactification $M \supseteq M_U$, with a Lagrangian fibration $\pi: M \to B$ extending $\pi_U$?
\end{question}

Here, by Lagrangian fibration we mean a proper surjective morphism from a holomorphic symplectic variety whose general fiber is a  Lagrangian subvariety (see Definition \ref{def-lagrfibr}). 
%In this paper, however, we will allow the total space to be a possibly singular symplectic variety (see Definition \ref{}), and in applications we will need to consider surjective morphisms whose general fiber is a proper Lagrangian submanifold, but which could be non proper.  
By the Arnold-Liouville Theorem, the general fiber of a Lagrangian fibration is an abelian variety. Examples of Lagrangian fibrations include Hitchin systems, elliptic K3 surfaces and their Hilbert schemes  of points, and the so-called Beauville-Mukai systems \cite{Beauville-intsystem, Mukai}. As a consequence of a remarkable theorem of Matsushita, Lagrangian fibrations are one of two classes of morphisms (with connected fibers) originating from a projective symplectic variety (the other class being birational morphisms, like symplectic resolutions or isomorphisms). Arguably for this reason, Lagrangian fibrations have been used pervasively to study (e.g. see \cite{Rapagnetta-topological, Rapagnetta10,Sawon-finiteness,  MRS,dCRS,HLS,Shen-Yin,Voisin-Lefschetz, ACLS, Debarre-Huybrechts-Macri-Voisin}, just to name a few) and construct \cite{ Markushevich-Tikhomirov,ASF, Matteini,LSV, BCGPSV} examples of \hk manifolds or (possibly singular) symplectic varieties.

Though we are able to answer the question positively in some cases (see Theorems \ref{thmlsv} and \ref{thm_smoothtorsors}), in this generality one should not always expect a positive answer, even if one relaxes the smoothness requirement and asks only for the existence of a singular holomorphic symplectic compactification (see Definition \ref{def-sympl} ). Indeed, there are some natural obstructions to the  existence of a holomorphic symplectic compactification. For example, if $M_U$ is uniruled, then the canonical bundle of a smooth projective birational model of $M_U$ cannot be effective; thus, $M_U$ cannot have a $K$-trivial compactification with canonical singularities. 
Another obstruction to the existence of a \hk compactification, which is of topological nature, was found in Brosnan \cite{Brosnan}. 
%This will be discussed in REF.

The first result of this paper is the following Theorem, which guarantees, under some assumptions, the existence of a singular  compactification, and gives a criterion for when such a compactification is smooth.

\begin{thmintro} \label{thm1} [=Theorem \ref{thm1'}]
Let $B$ be a normal quasi-projective variety and let $U \subset B$ be a non-empty open subset. Let $\pi_U: M_U \to U$ be a surjective morphism from a quasi-projective holomorphic symplectic manifold (or a $\Q$-factorial terminal symplectic variety) and suppose that $\pi_U$ is a Lagrangian fibration over a non-empty open subset of $U$. Denote by $\sigma_{M_U}$ the holomorphic symplectic form on $M_U$. Suppose that 
\begin{enumerate}
\item [(a)] $U$ contains all codimension one points of $B$;
\item [(b)] $\sigma_{M_U}$ extends to a  holomorphic form on a smooth compactification of $M_U$;
\end{enumerate}
Then there exists a $\Q$-factorial terminal symplectic variety $M$ containing $M_U$ as an open dense subset with a projective morphism $\pi: M \to B$ which extends $\pi_U$ and is a Lagrangian fibration.

Moreover, if $B$, and hence $M$, is projective, then the following are equivalent: 
\begin{enumerate}
\item  $M$ is smooth;
\item  $M$ is birational to a projective holomorphic symplectic manifold;
\item $M$ is smoothable by a flat deformation. 
\end{enumerate}
\end{thmintro}

A specific case of Theorem \ref{thm1} was used in \cite{IntJac} to show the existence of a \hk compactification of the intermediate Jacobian fibration associated with \emph{any} smooth cubic fourfold (and not only for the general one as proved in \cite{LSV}). In this paper, we extend that proof to the general case, we develop a framework for applying this theorem to geometric examples (see \S \ref{section-cycle}), and we apply it to study different Lagrangian fibrations that can be associated to a given Lagrangian fibration (see Theorems \ref{thm_torsors} and \ref{thm_smoothtorsors}).

Some comments on the assumptions are in order. First, note that we require the surjective morphism $M_U \to U$ to be projective only over the general point of $U$. Allowing non proper morphisms is crucial in some applications (e.g. in Theorem \ref{thm_torsors} and \ref{thm_smoothtorsors}) to guarantee the codimension $2$ assumption.

As we will see in Section \ref{section-cycle}, there are many abelian fibrations of \emph{geometric origin} which carry a holomorphic form that naturally  satisfies the extendability criterion of the theorem. A consequence of Theorem \ref{thm1} is that for these fibrations, it suffices to control that the holomorphic form is non-degenerate in codimension $1$ over $B$ to check that there is a $\Q$-factorial terminal symplectic compactification. Note that the Hitchin system does not satisfy the assumptions of the Theorem, since it is uniruled (indeed, it is acted upon by $\mathbb C^\times$).

A very interesting application of Theorem \ref{thm1} that fits into this context has recently been obtained by Liu-Liu-Xu in \cite{LLX}. In this paper, the authors show that the intermediate Jacobian fibration associated to the moduli space of smooth cubic fivefolds containing a general cubic fourfold admits a $\Q$-factorial terminal symplectic compactification which is, furthermore, an irreducible symplectic variety.  The second statement follows from a beautiful result they prove, showing that if a projective $\Q$-factorial terminal symplectic variety admits a Lagrangian fibration whose general fiber is a simple abelian variety, then this symplectic variety is irreducible (in the sense of the singular decomposition theorem). Another  application of Theorem \ref{thm1} will be given in the forthcoming \cite{MOGS}, where we study various Lagrangian fibrations associated to even Gushel-Mukai varieties.

 In this paper, the first, immediate, application of Theorem \ref{thm1} is a short proof of the main result of  \cite{LSV}.  We include this result because, together with Theorem \ref{thm_smoothtorsors}, these are the only two cases where we can show that the $\Q$-factorial terminal symplectic compactification produced using Theorem \ref{thm1} is  actually smooth.

For the purpose of this paper, the main applications of Theorem \ref{thm1} are to various quasi-projective Lagrangian fibrations that can be associated with a Lagrangian fibration that admits local (\'etale) sections over an open subset whose complement has codimension greater or equal to $2$. The set-up is the following. Given a Lagrangian fibration $X  \to B$, we let  $B^{nm} \subset B$ be the open subset parametrizing fibers with a reduced component (thus, $B^{nm}$ is the locus over which there exist local sections). As a starting point, we show that
there exists a smooth commutative algebraic group scheme $A \to B^{nm}$, and an action of $A$ on $X_{B^{nm}}$, making the pair $(X_{B^{nm}}, A)$ into a so-called $\delta$-regular weak abelian fibration (cf. \cite[\S 2.2]{dCRS}). We call such an $A \to {B^{nm}}$ the relative Albanese of $X_{B^{nm}} \to {B^{nm}}$. 
This is a result of independent interest which completes previous results by Markushevich \cite{Markushevich} and Arinkin-Fedrov \cite{Arinkin-Fedorov}. A consequence of this result is the fact that to a Lagrangian fibration admitting local sections, one can apply Ng\^o Support Theorem (and its refinement from \cite{dCRS}). The use of these powerful techniques in the study the topology of maps,  which have long been used to study Hitchin systems, has recently and successfully been introduced to study Lagrangian fibrations on compact \hk manifolds (e.g., see \cite{MRS,dCRS,Wu, Shen-Yin,ACLS}).

The second main result of the paper is to show that the compactification criterion of Theorem \ref{thm1} applies to any quasi-projective (almost) torsor over $A$ (see Definition \ref{defin-almost}), or over a smooth commutative group with connected fibers that is isogenous to $A$. In other words, any such torsor not only has a natural holomorphic symplectic form, but it also admits a $\Q$-factorial terminal symplectic compactification.
This is the content of Theorem \ref{thm_torsors}.
 A corollary of this result is that given a Lagrangian fibration that admits local sections in codimension $1$, we can construct a (possibly singular) Lagrangian fibration with a (rational) section that is (at least on an open subset) locally isomorphic to the original one (cf. Corollary \ref{cor-ratsect}). This answers a question asked by C. Voisin. Another corollary is to show the existence of a (possibly singular) symplectic compactification of (every component of) the relative Picard variety, at least in the case when this is represented in codimension $1$ by a separated scheme (cf. Corollary \ref{cor-pic}).

Theorem \ref{thm_torsors} also provides a new context in which to frame  the two exceptional examples of \hk manifolds. Indeed, both OG$6$ and OG$10$ can be constructed from Lagrangian fibrations that are torsors over (a group isogenous to) the relative Albanese of a Lagrangian fibration on a \hk manifold of type K3$^{[n]}$. See Examples \ref{ex-og10} and \ref{ex-og6}.

Finally, we show in Theorem \ref{thm_smoothtorsors} that when the fibers of $X \to B$ are integral, then the compactifications of $A$-torsors produced in Theorem \ref{thm_torsors} are smooth, and \'etale locally isomorphic to $X \to B$. As mentioned above, this is one of two cases when we can prove smoothness. We expect the \hk manifolds obtained this way to be deformation equivalent to the original $X$. As a starting point, we show they have isomorphic rational Hodge structures (cf. Proposition \ref{prop-coho}).

\subsection*{Relation to other work}
The papers \cite{Abasheva-Rogov, Abasheva} set the foundations for the study of Shafarevich-Tate twists of Lagrangian fibration on \hk manifolds. See also \cite{Markman-TS,Bogomolov-Kamenova-Verbitsky}. The starting point of  \cite{Abasheva-Rogov} is the consideration of the sheaf of vertical automorphisms of a Lagrangian fibration, which corresponds to the Albanese fibration considered in this paper. By definition, a $1$-cocyle with values in this sheaf of abelian groups determines a Shafarevich-Tate twist of a Lagrangian fibration. 
While they consider this sheaf under more general assumptions than we do in this paper, it is only considered as a sheaf of abelian groups and not as a algebraic group scheme.
In the case of integral fibers, the Lagrangian fibrations of Theorem \ref{thm_smoothtorsors} are Shafarevich-Tate twists of the original fibration.  These twists are easily constructed as complex manifolds but the challenge, answered in the two papers above, is to understand when they are K\"ahler and projective.

%
%Examples of \hk manifolds of OG$10$ and K3$^{[5]}$ that are Lagrangian fibrations over the same base,  are locally isomorphic over an open subset, but have fibers with a different number of components over some points indicate that the relation between Shafarevich-Tate twists and compactification of different torsors over the Albanese fibration should be investigated further.

The paper \cite{Kim-Neron} recently and independently also proves the existence of the relative Albanese variety of a Lagrangian fibration admitting local sections. The proof is very different, and rests on the elaborate and very interesting theory of N\'eron models of Lagrangian fibration developed by the author.

As mentioned in the introduction, the paper \cite{LLX} uses Theorem \ref{thm1}, as well as a subtle study of the moduli stack of pairs $(X,Y)$, where $X$ is a general smooth cubic fourfold and $Y$ is a cubic fivefold containing $Y$, to show that the intermediate Jacobian fibration associated to these pairs admits a compactification that is a symplectic variety. 
%Additionally, they prove that a Lagrangian fibration whose general fiber is a simple abelian variety is an irreducible symplectic variety in the sense of the decomposition theorem for singular $K$-trivial varieties.
Our hope is that Theorem \ref{thm1} can be applied to other example of integrable systems of geometric origin.

Finally, the recent paper \cite{Dutta-Mattei-Shinder} studies in great detail the Albanese fibration of the so-called LSV fibrations and computes the Tate-Shafarevich group parametrizing twists of the original Lagrangian fibration. I am grateful to E. Shinder for a correspondence on the group $\Lambda$ of Section \ref{section-Alb}.

\subsection*{Structure of the paper} In Section \ref{sect-proof} we prove Theorem \ref{thm1}(=Theorem \ref{thm1'}). In Section \ref{section-intjac} we reprove the main result of \cite{LSV} (see Theorem \ref{thmlsv}). In Section \ref{section-cycle} we introduce the notion of extendable holomorphic form and discuss how in many examples of geometric origin the holomorphic form is naturally extendable. In Section \ref{section-Alb}, we defined the relative Albanese fibration $A \to B$ of a Lagrangian fibration $X \to B$ admitting local sections over every point of the base, and we prove that it is a smooth commuative group scheme with connected fibers (Theorem \ref{thm alb}). We then show in Proposition \ref{prop_gamma} that $A$, as well as any (almost) $A$-torsor, has a naturally defined extendable holomorphic symplectic form. In Section \ref{section_torsors}, we apply the previous results to show that any (almost) torsor over $A$, or over a smooth commutative group isogenous to $A$ admits a holomorphic symplectic compactification (Theorem \ref{thm_torsors}) and that, if $X  \to B$ has integral fibers, the symplectic compactification of any $A$-torsor is smooth.

\subsection*{Acknowledgements} I am grateful to Claire Voisin for several conversations over the years that have shaped my understanding of  ``cycle theoretic'' holomorphic symplectic forms and for asking me about sections of Lagrangian fibrations. The point of view of Section \ref{section-cycle} and the proof of Proposition \ref{prop_gamma2} are indebted to these conversations. Theorem \ref{thm_smoothtorsors} and Corollary \ref{cor-ratsect} respond to her question.
I  thank Anna Abasheva, Yoonjoo Kim, and Evgeny Shinder for interesting conversations on Lagrangian fibrations, and Chenyang Xu for his interest in Theorem \ref{thm1}.
Some results in this paper, most notably Theorem \ref{thm_smoothtorsors} and Theorem \ref{thm1}, are a few years old and I have presented them at several seminars and conferences in the last few years. I thank the audience of these talks for their interest and their questions. Finally, I wish to thank Enrico Arbarello, Emanuele Macr\`i  and Claire Voisin for their encouragement in finishing this paper, and to Kieran O'Grady and Yoonjoo Kim for some comments on a draft of this paper. I am especially grateful to Anna Abasheva for finding two mistakes in the first version of this paper.
The author gratefully acknowledges partial support from NSF FRG grant DMS-2052750, from NSF CAREER award DMS-2144483, and from the Simons Collaboration ``Moduli of varieties''.

\section{Proof  of Theorem \ref{thm1}} \label{sect-proof}

In this section we will prove Theorem \ref{thm1'}, a slight generalization of Theorem \ref{thm1} which is needed for applications. We start by making a few definitions.

\begin{defin}  \label{def-sympl}
A normal quasi-projective variety $V$ is called a symplectic variety if the smooth locus $V^{sm}$ admits a holomorphic symplectic form $\sigma$ which extends to a holomorphic form on any resolution of $V$ \cite{Beauville-sympl-sing}.
\end{defin}

If $V$ is a symplectic variety with symplectic form $\sigma$ then, by definition, $\sigma$ is a (closed) reflexive $2$-form. A symplectic variety has rational Gorenstein singularities.  A $\Q$-factorial terminal symplectic variety is a symplectic variety which is $\Q$-factorial and has terminal singularities.
From the point of view of deformation theory view and moduli theory, symplectic projective varieties that are $\Q$-factorial and terminal   behave like smooth \hk manifolds \cite{Namikawa,Bakker-Lehn-global-moduli}.

\begin{defin} \cite[Definition 1.2]{Matsushita-higher} \label{def-lagrfibr}
Let $M$ be a symplectic variety with symplectic form $\sigma$. Let $B$ be a normal variety. A proper surjective morphism $f : M \to B$ with connected fibers is said to be a Lagrangian fibration if a general fibre $F$ of $f$ is a Lagrangian subvariety of $M$ with respect to $\sigma$.
That is, $\dim F = \frac{1}{2} \dim X$ and the restriction of $\sigma$ to $F \cap M^{sm}$ vanishes identically.
\end{defin}

In applications (see Theorems \ref{thm_torsors} and \ref{thm_smoothtorsors}, as well as Theorem \ref{thmlsv}, \cite{LLX}, and \cite{MOGS}), it will be important to consider the case when $M_U \to U$ is a surjective morphism that is a Lagrangian fibration over an open subset of $U$, but globally may not be proper. 
%Before proving the theorem we record the following consequence of a result of Nakayama (cf. Matsushita \cite[]{}).
%\begin{prop}
%Let $M_U \to U$, $U \subset B$ be as in Theorem \ref{thm1}. Then $B$ has log terminal singularities.
%\end{prop}
%\begin{proof} Let $\wt M$ be a smooth projective compactification of $M_U$ with a regular morphism $\wt \pi: \wt M \to B$.
%By \cite[Theorem 2]{Nakayama-sing}, it is enough to show that for every integer $\ell  \ge 1$, the natural morphism $i: (\wt \pi_* K_{\wt M})^{\otimes \ell} \to \wt \pi_* (K_{\wt M}^{\otimes \ell})$ is an isomorphism. Note that this is an isomorphism over $U$, since on $M_U$ the canonical class is trivial.
%By assumption, $K_{\wt M}$ is effective and supported on $\wt M \setminus M_U$, so $\mc O_B=\pi_* \mc O_{\wt M } \to \pi_* K_{\wt M}$ is an isomorphism in codimension $2$. Since by Koll\'ar's result $\pi_* K_{\wt M}$ is reflexive, which implies that $\pi_* K_{\wt M} \cong \mc O_B$. Hence $(\wt \pi_* K_{\wt M})^{\otimes \ell} $ is locally free, and hence $i$ is an isomorphism
%\end{proof}

\begin{thm} \label{thm1'}
Let $B$ be a normal quasi-projective variety and let $U \subset B$ be a non-empty open subset. Let $M_U$ be a $\Q$-factorial terminal symplectic variety with holomorphic form $\sigma_{M_U}$. Let $\pi_U: M_U \to U$ be a surjective morphism that is a Lagrangian fibration over a non-empty open subset of $U$.
Suppose that 
\begin{enumerate}
\item [(a)] $U$ contains all codimension one points of $B$;
\item [(b)] $\sigma_{M_U}$ extends to a  holomorphic form on a smooth compactification of $M_U$;
\end{enumerate}
Then there exists a $\Q$-factorial terminal symplectic variety $M$ containing $M_U$ as an open dense subset with a projective morphism $\pi: M \to B$ which extends $\pi_U$ and is a Lagrangian fibration.
Moreover conditions (1), (2), (3) as in Theorem \ref{thm1} are equivalent.
\end{thm}

\begin{proof} The proof of this theorem follows closely that of \cite[Thm 1.6]{IntJac}, which in turn uses results of Lai \cite{Lai}. We start with a few considerations on normal compactifications of $M_U$.

Assumption (b) implies that for any normal compactification $\ov M$ of $M_U$, the holomorphic form $\sigma_{M_U}$  extends to a holomorphic $2$-form on the regular locus of $\ov M$, i.e., $\sigma_{M_U}$ extends to a (closed) reflexive $2$-form $\sigma_{\ov M}$ on $\ov M$. 
Since this form is generically non-degenerate, the canonical class $K_{\ov M}$ of $\ov M$ is equal to the degeneracy locus $D:=D(\sigma_{\ov M})=\{\sigma_{\ov M}^{\dim M_U/2}=0\}$ of $\sigma_{\ov M}$. 
Thus it is effective, it is supported on $\ov M \setminus M_U$, and it is trivial if and only if $\sigma_{\ov M}$ is symplectic. 

Now suppose that there is a projective morphism $\ov \pi: \ov M \to B$ extending $\pi_U$. 
Since $M_U \to U$ is surjective with proper general fiber, it follows that if $D$ is non trivial, then by Assumption (a) either $\ov \pi (D)$ has codimension $ \ge 2$, in which case it is  called $\ov \pi$-\emph{exceptional}, or that $D$ is of \emph{insufficient fiber type}, i.e. that there exists a component $\Delta$ of $\ov \pi (D)$  and a component $\Gamma$ of ${\ov \pi}^{-1}(\Delta)$, such that $\Gamma \not\subseteq D$ (see \cite[\S 5.a]{Nakayama}). Now suppose $\ov M$ is $\Q$-factorial. 
Then by Lemma 2.10 of \cite{Lai} (cf. 
%Following \cite[Definition 2.9]{Lai}, we say that $D$ is degenerate. 
Lemmas 5.1 and 5.2 of \cite{Nakayama}),  if $D$ is non trivial, then there is a component of $D$ that is covered by $K_{\ov M}$-negative curves that are contracted by $\ov \pi$.

We construct the desired compactification of $M_U$ by choosing an arbitrary smooth projective compactification with a morphism to $B$, and then finding a suitable birational model.
Let $\wt M \supseteq M_U$ be a $\Q$-factorial partial compactification of $M_U$ with a  projective morphism $\wt \pi: \wt M \to B$ extending $\pi_U: M_U \to U$. Let $\sigma_{\wt M} \in H^0(\wt M, \Omega_{\wt M}^{[2]})$ be the holomorphic $2$-form extending  $\sigma_{M_U}$.  If $\sigma_{\wt M}$ is non-degenerate, then $\wt M$ is a $\Q$-factorial terminal symplectic variety  and we are done. 
If $\sigma_{\wt M}$ is degenerate then, by the discussion above, $K_{\wt M}$ is not $\wt \pi$-nef. In this case, we claim that we can use the minimal model program (mmp) to find a projective $\Q$-factorial terminal  birational model $M$ of $\wt M$  with a regular morphism $M \to B$ extending $\pi_U$ and whose canonical class  is nef over $B$. This implies that the holomorphic two form on the smooth locus of $M$ is non-degenerate, i.e. symplectic. Moreover, it extends to a holomorphic form on any resolution. Thus, the projective variety $M$ is  $\Q$-factorial terminal symplectic compactification of $M_U$.
%Otherwise, we claim that we can use the minimal model program (mmp) to find a birational model $M$ of $\wt M$  with a regular morphism $M \to B$ extending $\pi_U$ and  that is a $\Q$-factorial terminal compactification of $M_U$.  whose smooth locus has holomorphic symplectic two form and with a regular morphism $M \to B$ extending $\pi_U$. Since by construction this form extends to a holomorphic form on any resolution of $M$, it follows that $M$ is a $\Q$-factorial terminal symplectic projective variety.  

To prove the claim, we follow the proof of Theorem 1.6 of \cite{IntJac}. We recall here the main steps for the reader's convenience, referring to  \S 1.1 of \cite{IntJac}  for more details.
The key ingredient is the mmp with scaling. See \cite[\S 5.E]{Hacon-Kovacs} for a short introduction on this topic.

Let $H$ be a $\wt \pi$-ample $\Q$-divisor on $\wt M$ such that the pair $(\wt M, H)$ is klt and such that $K_{\wt M}+H$ is big and nef over $B$. The mmp with scaling of $H$ over $B$ produces a non-increasing sequence of non-negative rational numbers $1 \ge t_0 \ge \dots \ge t_i \ge \dots \ge 0 $ and a sequence  $\psi_i: M_i \dashrightarrow M_{i+1}$ of birational maps over $B$, which are defined inductively in the following way. See \cite[\S 5.E]{Hacon-Kovacs} or \cite[\S 1.1]{IntJac} for more details. 
Set  $ M_0=\wt M$ and $H_0=H$. Suppose $\phi_{i-1}$, $M_i$ and $H_i$ are defined and set 
\[
t_i=\inf\{ t \ge 0 \,\, | \,\, K_{M_i}+t H_i \text{ is nef over } B\}.
\]
If $t_i=0$, then $K_{M_i}$ is nef over $B$ and the process stops. 
Otherwise, by the Cone Theorem (\cite[Chapter 3]{Kollar-Mori}), there exists a $K_{M_i}$-negative ray $R_i$, which is contracted in $B$ and is such that $(K_{M_i}+t_i H_i)  \cdot R_i=0$.  Let $\psi_i: M_i \dashrightarrow M_{i+1}$ be the flip or the divisorial contraction associated with $R_i$ (the extremal contraction associated to $R_i$ cannot be a Mori fiber space since $M_i$ has effective canonical bundle). Since $R_i$ is contracted in $B$, the birational map $\psi_i$ is a map over $B$, in the sense that the morphism $\pi_{i+1}:=\pi_i \circ \psi_i^{-1}: M_{i+1} \to B$ is regular.
Since $\wt M$ is smooth, the $M_i$ are $\Q$-factorial. Set $H_{i+1}:=(\psi_{i})_* H_i$.
Since $K_{M_i}$ is effective, the fact that $R_i$ is a $K_{M_i}$-negative ray implies that any curve with class $R_i$ is contained in the support of $K_{M_i}$ and thus in $M_i \setminus M_U$. In particular, the map $\psi_i$ is an isomorphism over $M_U$ and hence $M_i$ is a $\Q$-factorial terminal compactification of $M_U$.

As a consequence of  \cite{BCHM}, the sequence of integers $t_i$ is eventually strictly decreasing  (see \cite{Druel-exc} or \cite[Lemma 1.13]{IntJac}). This means that if the process does not terminate, then $\lim t_i=0$. By construction, $\psi_i$ extracts no divisors (i.e. $\psi_i^{-1}$ contracts no divisors). By Lemma 1.19 of \cite{IntJac}, the number of components of the support of the effective divisor $K_{M_i}$ is eventually strictly decreasing. Thus, for $i \gg 0$, $K_{M_i}$ is trivial. As noticed above, this means that $M_i$ is a $\Q$-factorial terminal symplectic variety.

The equivalence of (1) and (2)  and of (1) and (3) follows from Proposition \ref{prop-smbir} below.
%The equivalence of (1) and (2) is a general fact about $\Q$-factorial terminal symplectic varieties that are projective over a common base, and is the content of \cite[6.4]{Greb-Lehn-Rollenske}. The equivalence between (1) and (3) follows from \cite[Main Theorem]{Namikawa-Def}.
\end{proof}

As a consequence of a theorem of Nakayama (see Matsushita \cite{Matsushita-fiber}), assumption (b) forces the  singularities of $B$ to have log terminal singularities. In practice, in many applications of interest $B$ will be closely related to $\mathbb P^n$.
Note also that as a consequence of Theorem \ref{thm1} (3), the existence of a \hk compactification of $M_U$ is equivalent to the existence of a \hk compactification of $M_U$ with a Lagrangian fibration, so adding this condition is not more restrictive.

As observed by \cite[6.4]{Greb-Lehn-Rollenske}, the following proposition is a consequence of work of \cite[Main Theorem]{Namikawa-def} and Kawamata \cite{Kawamata-flops}.

\begin{prop} \label{prop-smbir}
Let $M$ be a projective $\Q$-factorial terminal symplectic variety. If $M$ is birational to a smooth projective \hk manifold, or $M$ can be deformed in a flat family to a smooth projective \hk manifold, then $M$ is smooth.
\end{prop}
\begin{proof}
The first statement is the  content of \cite[6.4]{Greb-Lehn-Rollenske}. The second follows from \cite[Main Theorem]{Namikawa-def}.
\end{proof}

\begin{rem}
The same proof, mutatis mutandis, works if instead of starting from a smooth $M_U$, we start from a $\Q$-factorial terminal symplectic variety with a holomorphic reflexive form $\sigma_{M_U}$ which extends to a holomorphic form on a smooth projective model of $M_U^{sm}$.
\end{rem}

The following example (cf. \cite[Remark 3]{Matsushita-almost}) shows that one cannot in general relax the condition that $U$ contains the codimension one points of $B$, without modifying the base.

\begin{example} 
Let $X \to \P^2$ be a Lagrangian fibered \hk fourfold. Let $g: \P^2 \dashrightarrow \P^2$ be a Cremona transformation that is not regular and let $V \subset \P^2$ be the largest open subset where $g$ is an isomorphism. Set $U=g(V) \subset \P^2$. Then the holomorphic symplectic form on $M_U:=X \times_VU \to U$ extends to a regular form on any smooth compactification of $M_U$, but there is no smooth projective \hk compactification of $M_U$ with a Lagrangian fibration over $B$ extending $M_U \to U$ compatibly with the natural open embedding $U \subset B$.
\end{example}

\section{Cycle theoretic holomorphic symplectic form} \label{section-cycle} 

The following definition is crucial for the applications of Theorem \ref{thm1'} discussed in this paper.

\begin{defin} \label{def-ext}
A holomorphic form on a smooth quasi-projective variety $V$ is called \emph{extendable} if it is the restriction of a holomorphic form  defined on a smooth projective compactification of $V$.
A reflexive form on a normal variety $W$ is called extendable, if the holomorphic form on $W^{sm}$ is extendable.
\end{defin}

An extendable holomorphic form is always closed. Since the space of holomorphic forms on smooth projective varieties is a birational invariant, it suffices to check that a form extends on one smooth projective compactification to guarantee it extends to all smooth projective compactifications.  See  \cite[IIa]{Griffiths-variations-Abel} or \cite{Kebekus-Schnell} for extendibility criteria in terms of local integrability of forms. Here, we need the following Hodge theoretic characterization of closed extendable holomorphic forms, which implies that the space of closed extendable form is preserved by morphisms of mixed Hodge structures as well as by correspondences of the appropriate degrees.

\begin{rem} \label{remextsympl} 
\begin{enumerate}
\item Let $V$ be a normal quasi-projectice variety with a reflexive holomorphic $2$-form that is extendable and non-degenerate (i.e., its top wedge product is nowhere vanishing). Then $V$ is a symplectic variety.
\item Let $V$ be a projective variety. Then $V$ is a symplectic variety if and only if  it admits a reflexive holomorphic $2$-form that is extendable and non-degenerate.

\end{enumerate}
\end{rem}

For later use we record the following

\begin{lemma} \label{lemma-extsympl}
Let $V$ be a quasi-projective symplectic variety with an extendable symplectic form $\sigma$ (i.e. the corresponding reflexive form is extendable). Then $\sigma$ defines a generically non-degenerate reflexive extendable holomorphic $2$-form on every normal birational model $V'$ of $V$.
\end{lemma}
\begin{proof}
Let $\wt V$ be a smooth projective compactification of $V^{sm}$ and let $\sigma_{\wt V}$ be the corresponding holomorphic $2$-form which exists because $\sigma$ is extendable. Let $\wt V'$ be a a smooth projective compactification of ${V'}^{sm}$. Since $\wt V'$ is birational to $\wt V$, $\sigma_{\wt V}$ defines a generically non-degenerate holomorphic $2$-form on $\wt V'$ and hence by restriction a generically non-degenerate extendable holomorphic $2$-form on ${V'}^{sm}$, i.e., an extendable reflexive form on $V'$ that is generically non-degenerate.
\end{proof}

\begin{prop} \label{propextmhs}Let $V$ be a smooth quasi-projective variety.
A closed holomorphic $k$-form $\alpha$ on $V$ is extendable if and only if the class of $\alpha$ in $H^k(V, \mathbb C)$ lies in $W_0 \cap F^k$. 
\end{prop}
Here,  $W_* \subset H^k(V, \mathbb C)$ and  $F^* \subset H^k(V, \mathbb C)$ are respectively the weight and the Hodge filtration of the natural mixed Hodge structure (MHS) on  $ H^k(V, \mathbb C)$.

\begin{proof} Let $\wt V$ be a smooth projective compactification of $V$.
The holomorphic $k$-form $\alpha$ on $V$ is extendable if and only if there exists a holomorphic $k$ form $\wt \alpha \in H^0(\wt V, \Omega^k_{\wt V})$ such that $\wt \alpha_{|V}=\alpha$. By strictness of morphisms of MHS, and the fact that $\im (H^k(\wt V, \mathbb C) \to H^k(V,\C))=W_0 \subset H^k(V, \mathbb C)$, it follows that $\alpha$ is extendable if an only if it  is closed and the class $[\alpha] \in H^k(V, \mathbb C)$ lies in $W_0 \cap F^k$.
\end{proof}

%\begin{cor} Let $Z$ be a smooth projective variety and $V$ be a smooth quasi-projective. For any $\Gamma \in CH^*(Z)$
%The space of extendable holomorphic forms on a smooth quasi-projective variety is preserved by
%\end{cor}

The following set-up exemplifies what we mean by ``holomorphic symplectic form of geometric origin'': a holomorphic symplectic form that is defined, via a correspondence, from a Hodge structure of K3 type on a smooth projective variety.

\subsection*{Set up}: Let $X$ be a smooth projective variety of dimension $m$. Suppose that $H^{2\ell}(X)$ is a Hodge structure of K3-type, i.e.,  there are only three non-zero Hodge numbers and 
\[
h^{p,q}=0, \, \text{ for } (p,q) \neq (\ell+1,\ell-1), (\ell,\ell), (\ell-1,\ell+1), \, \text{ and } 
h^{\ell+1,\ell-1}(X)=h^{\ell-1,\ell+1}(X)=1.
\] 
Let $V$ be a smooth quasi-projective variety.  Consider a cycle $\Gamma \in CH^j(V \times X)$\footnote{Chow groups are considered with rational coefficients.} with $\ell+j-m=1$. Since $X$ is projective, $\Gamma$ induces a morphism  
\[
\Gamma^*: H^{2\ell}(X) \to H^2(V), \quad  \alpha \mapsto (p_V)_* (p_X^*(\alpha) \cap [\Gamma]),
\]
where $[\Gamma] \in H^{2j}(V \times X)$ is the cycle class of $\Gamma$.

%(Should I do this also for $\ell+j-m=i$, so $\Gamma^*: H^{2\ell}(X) \to H^{2i}(V)$ and get other holomorphic forms, or even canonical forms?)

\begin{prop}  \label{prop_cor_ext} Let $\eta$ be the generator of $H^{\ell+1, \ell-1}(X)$. Then $\Gamma^*(\eta)$ is an extendable holomorphic $2$-form on $V$.
\end{prop}
\begin{proof}
Let  $\wt V$ be a smooth projective compactification of $V$ and let $\wt \Gamma$ be a lift of $\Gamma$ to $CH^j(\wt V \times X)$ (for example the closure in $\wt V \times X$ of $\Gamma$. Then $\wt \Gamma^*: H^{2\ell}(X) \to H^2(\wt V)$ extends $\Gamma^*$ and is a morphism of Hodge structures of bidegree $(j-m,j-m)$. Hence, $\wt \Gamma^*(\eta)$ is a holomorphic two form.
\end{proof}

Of course the Hodge number $h^{\ell+1,\ell-1}$ need not be $1$-dimensional for $\wt \Gamma^*(\eta)$ to be a holomorphic $2$ form. We restrict to this case since our goal is to construct symplectic varieties with a unique symplectic form.

%%\begin{rem}
%%In general there is no reason for $\wt \Gamma^*(\eta)$ to be a symplectic form and in fact it may well be identically zero. However, in the case when $V$ is a Lagrangian fibration, the techniques of Donagi-Markman \cite{Donagi-Markman-cubic} and \cite[\S 1]{LSV} give a way to test the non-degeneracy of a holomoprhic form with respect to which the fibers are isotropic. CF REMARK DOPO
%%\end{rem}

\begin{example} [$m=2\ell=j=2$] Let $X$ be a K3 surface with symplectic form $\eta_X$ and let $V$ be the stable locus of an irreducible component of the moduli space of semistable objects in $D^b(X)$, with respect to some (Bridgeland) stability condition. Let $\Gamma \in CH^2(V \times X)$ be the degree $2$ component of the Chern character of a quasi-universal family. Then $\Gamma^*(\eta_X)$ is an extendable holomorphic $2$-form. In fact, it can be shown that this form is (up to a scalar)  the same one constructed by Mukai (see  \cite{OGrady-Donaldson, OGrady-SantaCruz}, or also \cite[Chap. 10.3]{Huybrechts-Lehn}). Viewing Mukai's form this way has the advantage of providing an easy proof that it closed. By Remark \ref{remextsympl}, this point of view also shows that the singular moduli spaces are symplectic varieties in the sense of Beauville (a result that was usually proved invoking a non-trivial result of Flenner).
\end{example}

\begin{example} ($m=2\ell=4, j=3$). Let $X$ be a smooth Fano $4$fold of K3-type, i.e., $H^4(X)$ is a Hodge structure of K3-type. Examples include smooth cubic fourfolds and Gushel-Mukai fourfolds. Let $\eta_X$ be the generator of $H^{3,1}(X)$. For any $V$ smooth quasi projective variety and any $\Gamma \in CH^3(V \times X)$, $\Gamma^*(\eta)$ is an extendable holomorphic two form on $V$. Similarly, a parameter space of sheaves (or complexes of sheaves) on $X$ has a holomorphic $2$-form. As showed in
\cite{Kuznetsov-Markushevich}, if the objects parametrized lie in the Kuznetsov component of $X$, then this form is non-degenerate.
This recovers the holomorphic symplectic form on all the known examples of hyper-Kahler manifolds and holomorphic symplectic varieties associated to $X$, see \cite{Donagi-Markman,Beauville-Donagi,Kuznetsov-Markushevich,LLSvS,LSV,Perry-Pertusi-Zhao,BLMNPS} among others. %the holomorphic $2$-form constructed in \cite{Kuztensov-Markushevich} on the stable locus of a moduli space of stable bundles (or more generally in $X$
\end{example}

In \cite{LSV} we developed a way to construct extendable holomorphic $2$-forms on families of intermediate Jacobian fibrations. The following two examples summarize the two main setups that are considered in  \cite[\S 1]{LSV}  and give some geometric examples they apply to.

\begin{example}  (\cite[\S 1]{LSV}). \label{exhyperplane}
 ($m=2\ell$, $j=\ell+1$.) Let $X$ be as above, and let $L$ be a line bundle on $X$. Suppose the smooth members $Y$ of $|L|$ satisfy 
\be \label{vanishing}
h^{p,q}(Y)=0, \,\, \text{ for } \,\, p+q = 2k-1 \,\, \text{ and } \,\, (p, q) \notin \{(k, k-1), (k-1, k)\}
\ee
for some integer $k$. Suppose also that  $\eta \in H^{\ell+1,\ell-1}(X)$, is such that $\eta_{|Y}=0 \in H^{\ell+1,\ell-1}(X)$. Let $U \subseteq |L|$ be the locus parametrizing smooth hyperplane sections, and let $V=J_U \to U$ be the family of intermediate Jacobians of the smooth members of $|L|$. This is a projective morphism. Suppose $k=\ell$. By Lemma 1.1 and Theorem 1.2 of \cite{LSV} there exists a cycle $\mathcal Z \in \CH^{k+1}(\mc J_U \times X)$ such that the class $\mathcal Z^* (\eta)$ is an extendable holomorphic $2$-form with respect to which the fibers of $J_U \to U$ are isotropic. Note that if $\mathcal Z^* (\eta) \neq 0$, then $\mc Z^*$ defined an embedding of $H^{2\ell}_{tr}(X)$ in $H^2(\wt V)$, where $\wt V$ is any smooth projective compactification of $V$. The following are some examples fitting in this set-up
\begin{enumerate}
\item  \label{excubic} The case when $X$ is a general smooth cubic fourfold is the content of \cite{LSV}. In Proposition 1.6 of loc. cit. it is proved that the extendable holomorphic $2$-form is non degenerate on $J_U$. In Proposition 1.9 of loc. cit, it is proved that this form is non-degenerate on the partial compactification of $J_U$ defined by adding the generalized Jacobian of  hyperplane sections with one $A_1$ singularity;
\item \label{exgm6} The case when $X$ is a Gushel-Mukai $6$-fold is the subject of forthcoming work \cite{MOGS}. In this case, the form degenerates along a divisor. This fibration is closely related to that of Example \ref{exfamily} (\ref{exgm4}).
\end{enumerate}
\end{example}

\begin{example}  (\cite[\S 1]{LSV}) \label{exfamily}
 By Theorem 1.4 of \cite{LSV}, the example above generalizes as follows: $\mc Y_U \to U$ is a family of smooth projective varieties satisfying the cohomological vanishing of (\ref{vanishing}) with respect to an integer $k$, and parametrized by a smooth base;  there exists a cycle $ \mc Z \in \CH^{m-\ell+k}( \mc  Y_U \times X)$. Then there exists an extendable holomorphic form on the intermediate Jacobian fibration $\mc J_U \to U$ with respect to which the fibers are isotropic. 

\begin{enumerate}
\item As showed in Example 1.5 of \cite{LSV}, this applies to the Illiev-Manivel integrable system, i.e. to the family of smooth cubic fivefolds containing a fixed cubic fourfold $X$. It is proved in \cite{Iliev-Manivel} (cf. also \cite{Markushevich-integrableJac}) that this extendable holomorphic two form is non degenerate on $\mc J_U$. In \cite{LLX}, after a careful analysis of the base of the Lagrangian fibration, which is the  (good) moduli space of cubic fivefolds containing a general $X$,  it is proved that the form is non-degenerate in codimension one. Theorem \ref{thm1} then gives a $\Q$-factorial terminal symplectic variety.

\item \label{exgm4} The set-up above applies when $X$ is an ordinary Gushel-Mukai fourfold and $Y_U \to U$ is the family of smooth Gushel-Mukai $5$-folds containing $X$, giving a generically non-degenerate form. This fibration is  also the subject of forthcoming work \cite{MOGS} (cf.  Example \ref{exhyperplane} (\ref{exgm6})).

\item The set-up above applies also when $X$ is a \hk manifold, $\mc Y_U \to U$ is a family of Lagrangian submanifolds, and  when $\mc J_U \to U$ is the  relative Albanese or Picard variety of $\mc Y_U \to U$. This recovers the holomorphic $2$-form of Donagi-Markman \cite{Donagi-Markman}.

\end{enumerate}

\end{example}

\begin{example} The case $m=j=2n, \ell=1$ is the case of families of zero cycles on holomorphic symplectic varieties, and is the most relevant to this paper (cf. Proposition \ref{prop_gamma} below).
\end{example}

In general, the holomorphic $2$-forms constructed above need not be even generically non-degenerate. A good feature of fibrations is that they are particularly well suited to check if the form is non-degerate, as the following useful criterion shows. This was already used in the pioneering work of Donagi-Markman \cite{Donagi-Markman}, as well as in \cite[\S 1.4]{LSV}.

%%\begin{rem}
%%In general, the holomorphic $2$-forms constructed above need not be even generically non-degenerate. A good feature of Lagrangian fibrations is that they are particularly well suited to check if the form is non-degerate, thanks to a criterion of Donagi-Markman \cite{}. In \cite{LSV} these techniques where developed to prove that the holomorphic $2$-form in the case of the intermediate Jacobian fibration of the hyperplane sections of a smooth cubic fourfold is non-degenerate. Using these same computations LLX, MOGS.
%%\end{rem}
%%
%%The following is a useful criterion for checking whether a holomoprhic $2$-form on a fibration is non-degenerate. This was already used in the pioneering work of Donagi-Markman \cite{Donagi-Markman}, as well as in \cite[\S 1.4]{LSV}.

\begin{prop} \label{prop_nondeg}
Let $p: P \to U$ be a smooth surjective morphism of smooth quasi-projective varieties. Suppose that there exists a holomorphic $2$-form $\sigma$ on $P$ with respect to which the fibers of $p$ are isotropic. Then $\sigma$ is non-degenerate if and only if the induced morphism $T_{P/U} \to p^* \Omega^1_U$ is an isomorphism.
\end{prop}
\begin{proof}
Contraction with the holomorphic $2$-form $\sigma$ determines a morphism $T_P \to \Omega_P$. Precomposing with the inclusion $T_{P/U} \subset T_P$ of the relative tangent bundle and postcomposing with the natural surjection $\Omega_P \to \Omega_{P/U}$ gives a morphism of vector bundles $T_{P/U} \to \Omega_{P/U}$, which is zero because the fibers of $p$ are isotropic. Hence there are morphism 
$\phi: T_{P/U} \to p^* \Omega^1_U$ and $\psi: p^* T_U \to \Omega_{P/U}$ which are (up to a sign) dual to each other. Equivalently, the filtration on $\Omega^2_P$ determined by the short exact sequence $0 \to p^* \Omega_U \to \Omega_P \to \Omega_{P/U} \to 0$ has graded pieces isomorphic to $\Omega^2_{P/U}$, $p^* \Omega_U \otimes \Omega_{P/U}$, and $p^* \Omega_U^2$. If the fibers of $p$ are isotropic with respect to $\sigma$, it follows that $\sigma$ defines an element $\sigma^{1,1} \in H^0(P, p^* \Omega_U \otimes \Omega_{P/U})=H^0(U,\Omega_U \otimes p_* \Omega_{P/U}) $. The morphism $\phi$ is induced by $\sigma^{1,1}$.
Note that $\sigma$ is non-degenerate if and only if $T_P \to \Omega_P$ is an isomorphism and this happens if and only if $\phi$ is an isomorphism (and that this property does not change if we multiply $\sigma$ by an element of $p^* \Omega_U^2$).
\end{proof}

A crucial ingredient of Propositions 1.4 and 1.9 \cite{LSV} is to interpret the fiber of the morphism $\phi$ at a given point $u \in U$ in terms of the geometry of the corresponding variety $Y_u$. This is also how non-degeneracy is checked in \cite{LLX} and in \cite{MOGS}.

\section{Applications I: intermediate Jacobian fibrations} \label{section-intjac}

%In this section, we give a some applications of Theorem \ref{thm1} to intermediate Jacobian fibrations.

As a first application, we reprove the main result of \cite{LSV}, namely that the Donagi-Markman integrable system has a smooth \hk compactification of OG$10$-type.
%i.e. the existence of a smooth \hk compactification of the Donagi-Markman integrable system and the fact that this \hk manifold is of OG$10$-type. 
We should stress that, while the proof we give here does not use the relative Prym varieties, which were used in \cite[\S 5]{LSV} to construct a compactification and show that it is smooth,  it does in fact use the results from \cite[\S 1]{LSV}, namely the existence of a partial compactification with an extendable holomorphic symplectic form, which guarantee that the assumptions of Theorem \ref{thm1'} are satisfied.
Moreover, while the present proof is  shorter than the one of \cite{LSV}, it is highly non-constructive. In particular, the methods used here do not say anything about the boundary of the compactification.
% (Note that the description of the singular fibers can often be useful, see \cite{HRS} where this description was used to give a short computation of the topologicalEuler characteristic of these \hk manifolds.)

Note that  Theorem \ref{thmlsv} and Theorem \ref{thm_smoothtorsors} below are the only two cases where we can prove the the $\Q$-factorial terminal symplectic compactification produced by Theorem \ref{thm1'} is smooth.

\begin{thm}[\cite{LSV,Voisin-twisted,IntJac}] \label{thmlsv} Let $X\subset \mathbb P^5$ be a smooth cubic fourfold, let $U \subset  (\mathbb P^5)^\vee$ be the locus parametrizing smooth hyperplane sections and let $J_U \to U$ be the family of intermediate Jacobians (or its twisted version of \cite{Voisin-twisted}). Then there exists a \hk compactification of $J_U$. This \hk manifold is a projective $10$-fold, deformation equivalent to O'Grady's $10$-dimensional example.
\end{thm}
\begin{proof}
First note that, as showed in \cite[\S 1.4]{LSV} (cf. also \cite{Donagi-Markman}), there is a partial extension $J_{U_1}$ of the fibration to the locus $U_1 \subset \mathbb P^5$ parametrizing hyperplane sections that are smooth or have at most one $A_1$ singularity. 
%The morphism $J_{U_1} \to U_1$ is not proper, since
The fiber over a singular cubic threefold is an extension of the Jacobian of a certain genus $4$ curve by $\mathbb C^\times$. By \cite[\S 1 ]{LSV} (cf. Example \ref{exhyperplane} (\ref{excubic})), there is an extendable holomorphic $2$-form on $J_{U_1}$ and this form is non-degenerate by \cite[Prop. 1.6 and 1.9]{LSV}. Hence by Theorem \ref{thm1'}, there is a $\Q$-factorial terminal symplectic compactification  $\bar J_X$  of $J_{U_1}$. 
The bulk of the proof is to show that $\bar J_X$ is smooth, which we do by using the criterion of Proposition \ref{prop-klsvplus} below.
%By the equivalence of (1) and (3) in Theorem \ref{thm1'},to show that $\bar J_X$ is smooth, it is enough to show that it is smoothable.

Let $\mc X \to T$ be a family of smooth cubic fourfolds parametrized by a smooth quasi-projective curve $T$. 
Let $\mc U \subset (\mathbb P^5_T)^\vee$ be the open subset parametrizing smooth hyperplane sections  of the members of the family and let $\pi: \mc J_{U} \to \mc U$ be the corresponding intermediate Jacobian fibration. 

Let $\wt {\mc J}$ be a smooth partial compactification of $\mc J_{U}$ with a projective morphism to $T$. The general fiber of $\wt {\mc J} \to T$ is a smooth projective variety that is birational to the $\Q$-factorial terminal symplectic variety $\bar J_{X_t}$ and every fiber has a reduced component that is birational to $\bar J_{X_t}$.
By Proposition \ref{prop-smbir} and Proposition \ref{prop-klsvplus} below, it suffices to show that there is one $t_0$ such that $\bar J_{X_{t_0}}$ is smooth to prove that  $\bar J_{X_t}$ is smooth for every $t$.
To show this, we consider another family of cubic fourfoulds, namely, a pencil $\mc X_S \to S$ of cubic fourfolds degenerating to the chordal cubic. Recall that the chordal cubic is the singular cubic hypersurface in $\P^5$ isomorphic to $\Sym^2(\mathbb P^2)$. Up to passing to an open subset of $S$, we can assume that for $s \neq o$, $\mc X_s$ is smooth and for $s=o$, $\mc X_o$ is the chordal cubic. By Lemma 4.4 of \cite{IntJac} there is a projective morphism $\mc J_S \to S$ from a smooth quasi-projective variety such that for general $s$, $\mc J_s$ is birational to $\bar J_{\mc X_s}$ and such that for $s=o$, the fiber $ \mc J_o$ is birational to an OG$10$ \hk manifold. 
%Let $\wt{\mc J}_S$ be a smooth  partial compactification of $\mc J_{S}$ with a projective morphism to $S$.
By Proposition \ref{prop-klsvplus}, up to a base change, there is a birational model $\mc J'_S$ of ${ \mc J}_S$ over $S$ such that all fibers $\mc J'_s$ are projective symplectic varieties,  for $s \neq o$,  $\mc J'_s$  is birational to $\bar J_{\mc X_s}$, and for $s=o$, $\mc J'_o$ is birational to a smooth projective \hk manifold. By Proposition \ref{prop-sm-res} below, $\bar J_{\mc X_s}$ is smooth for every $s$.

\end{proof}

In the proof of the Theorem above, we used the following two general results of independent interest. The first is a variation on \cite[Thm 2.1]{KLSV}:

\begin{prop} \label{prop-klsvplus}
Let $l:  Y \to T$ be a projective morphism to a smooth quasi-projective curve. Suppose that the general fiber is irreducible and that for every $t \in T$ there is a reduced component of the fiber $Y_t$ that is birational to a projective symplectic variety $M_t$.
%
%\begin{enumerate}
%\item[(a)] for general $t$ the fiber $ Y_t$ is birational to a projective symplectic variety,
%\item[(b)] for every $t$ there is a reduced component of $ Y_t$ that is not uniruled.
%\end{enumerate}
Then up to a base change $p: T' \to T$, there exists a projective morphism $l':  Y' \to T'$ with the following properties: 
\begin{enumerate}
\item $Y'$ is birational over $T$ to $Y\times_T T'$,
\item the fibers of $l'$ are irreducible and are projective symplectic varieties (additionally, the general fiber has terminal singularities),
\item for every $t'\in T'$, $Y'_{t'}$ is birational to $M_t$.
\end{enumerate}
Moreover, if for one $t' \in T$, the fiber $ Y'_t$ is birational to a smooth projective \hk manifold $M$ or can be deformed in a flat family to a smooth projective \hk manifold $M$, then for all $t \in T'$  the symplectic variety $ Y'_t$ admits a symplectic resolution $\wt Y_t \to  Y'_t$. If this is the case, then for every $t$, the smooth projective \hk manifold $\wt Y_t $ is deformation equivalent to $M$.
%\(a) the fibers of $l'$ are irreducible (b) $Y'$ is birational over $T$ to $Y\times_T T'$ (c) for every $t'$, $\mc Y'_{t'}$ is a symplectic variety and it is birational to the non uniruled component of $\mc Y_{p(t')}$ (c) for general $t$, $\mc Y'_t$ is terminal. 
%Moreover, if for one $t$, $\mc Y'_t$ is birational to a smooth projective \hk manifold $M$, or can be smoothed in a flat family to a smooth projective \hk manifold $M$, then $\mc Y'_t$ admits a symplectic resolution $\wt Y_t \to \mc Y'_t$ for all $t \in T'$. Moreover, for every $t$, the smooth projective \hk $\wt Y_t $ is deformation equivalent to $M$.
\end{prop}
\begin{proof}
By Theorem 2.1 of \cite{KLSV},  up to a base change $p: T' \to T$, there exists a projective morphism $ Y' \to T' $ such that all fibers are irreducible $K$-trivial varieties with canonical singularities (and the general fiber has terminal singularities). Moreover, $Y'_{t'}$ is birational to the component of $ Y_{p(t')}$ that is birational to $M_t$. By Remark \ref{remextsympl} and Lemma \ref{lemma-extsympl},  $ Y'_t$ has a reflexive extendable holomorphic $2$-form that is generically non degenerate. Since $ Y'_t$ has trivial canonical class, the  extendable holomorphic $2$-form is actually non-degenerate and hence $ Y'_t$ is a symplectic variety. 
%To show that $ Y'_t$ is a symplectic variety for every $t$, we note that the results of \S 5 of \cite{KLSV} go through  when the general fiberthe proof of Theorems 1.3 and 1.7
The final statement now follows from Proposition \ref{prop-sm-res} below.
\end{proof}

\begin{prop} \label{prop-sm-res}
Let $M$ be a projective symplectic variety. If $M$ is birational to a smooth projective \hk manifold, or $M$ can be deformed in a flat family to a projective symplectic variety that admits a symplectic resolution, then $M$ admits a symplectic resolution.
\end{prop}
\begin{proof}
By \cite[Cor. 1.4.3]{BCHM}, $M$ admits a $\Q$-factorial terminalization, i.e., a crepant birational projective morphism $t: \wt M \to M$ such that $\wt M$ is $\Q$-factorial and terminal. Since $t$ is crepant, $\wt M$ is a symplectic variety. 

Suppose that $M$ is birational to a smooth projective \hk manifold. Then so is $\wt M$, and hence by Proposition \ref{prop-smbir} it is smooth and $t: \wt M \to M$ is a symplectic resolution.

Suppose that $\mc M \to T$ is a flat family,  that $M=\mc M_{t_1}$ for some $t_1 \in T$, and  that $N:=\mc M_{t_2}$ admits a symplectic resolution $\wt N \to N$. For a projective variety $Y$, denoted by $\Def(Y)$ the Kuranishi or deformation space of $Y$. By Theorem 2.2 of \cite{Namikawa-def}, $\Def(\wt N) $ and $ \Def(N)$ are smooth of the same dimension, there is a finite surjective morphism $\Def(\wt N) \to \Def(N)$, compatible with the universal families and such that $N$ can be smoothed by a flat deformation. As a consequence $M$ can be smoothed by a flat deformation. By \cite[Cor. 2]{Namikawa}, $\wt M$ is smooth.
%By \cite[Thm. 1]{Namikawa-Qfact}, $\Def(\wt M)$ and $\Def(M)$ are both smooth of the same dimension, and there is a finite surjective morphism $\Def(\wt M) \to \Def(M)$, compatible with the universal families (in the sense that $M$ and $\wt M$ have a common deformation). Since $M$ can be smoothed by a flat deformation, so can $\wt M$. By \cite[Cor. 2]{Namikawa-Qfact}

\end{proof}

%%
%%A second application has been recently been carried out in the paper \cite{LLX}. EXPAND
%%
%%A third application is the case of the GM 4fold, GM 6folds. This is the case of upcoming work \cite{}.

\section{The relative Albanese} \label{section-Alb}

The aim of this section is to construct the relative Albanese fibration associated to a Lagrangian fibration, which will be used in Section \ref{section_torsors}. The existence of this fibration extends results by Markushevich \cite{Markushevich} and Arinkin-Fedorov \cite{Arinkin-Fedorov} and is a result of independent interest. The proof we give here amounts to the Arnold-Liouville theorem and is essentially due to Markushevich. See Remark \ref{rem albanese literature} below for more precise remarks on the literature concerning this group.

Let $f: X \to B$ be a Lagrangian fibration. Following \cite{Markushevich}, we introduce the following notation: We let $X^{nc} $ be the non-critical locus of the map $f$, i.e., the open subset of $X$ where the rank of the differential of $f$ is maximal. We set $B^{nm}=f(X^{nc}) \subset B$ to be the open subset of $b \in B$ such that $X_b$ is non-multiple, i.e. such that $X_b$ has a reduced component. Then $B^{nm}$ is also the largest open subset of $B$ over which there exist (\'etale) local sections.
Since both $X^{nc}$ and the morphism $X^{nc} \to B^{nm}$ are smooth, so is $B^{nm}$.

\begin{thm} \label{thm alb}
Let $f: X \to B$ be a Lagrangian fibration, with $f$ projective. Then there exists a separated smooth commutative algebraic  group scheme $a: A \to B^{nm}$ with connected fibers, with an action $A \times_{B^{nm}} X_{B^{nm}} \to X_{B^{nm}} $. Moreover, the pair $(X_{B^{nm}} , A)$ is a $\delta$-regular weak abelian fibration and the action of $A$ on $X_{B^{nm}} $ makes $X^{nc} \to {B^{nm}} $ into an almost $A$-torsor (Definition \ref{defin-almost}).
%(this means that $X^{nc}$ is covered by  $A$-invariant open subsets that are $A$-torsors; IS THIS EQUIV?? to asking that thechoice of an \'etale section $V \to ,X^{nc}$, makes the $A$-variety $X^{nc,*}$ into the trivial  $A$-torsor. YJ Kim: have thee trivial stabilizer, but not the transitivity of the action).
\end{thm}

Recall that a  weak abelian fibration is a pair $(M,P)$ of schemes over $S$ of the same pure relative dimension, where $M \to S$ is a proper morphism, $P \to S$ is a smooth commutative group scheme with polarizable Tate modules acting on $M$ with affine stabilizers. Such a triple is called $\delta$-regular if, furthermore, the rank of the affine part of the group $P_s$, $s \in S$, jumps by $\delta$ in codimension $\ge \delta$. We refer the reader to \cite[\S 2.2]{dCRS} for more details.

\begin{defin} \label{def-rel-alb}
Let $f: X \to B$ be a Lagrangian fibration, with $f$ projective and with the property that $B=B^{nm}$. The algebraic group scheme $a: A \to B$ of Theorem \ref{thm alb} is called the Albanese family or the relative Albanese fibration of the Lagrangian fibration $f: X \to B$. \end{defin}

\begin{rem} \label{rem albanese literature}
The existence of the relative Albanese fibration acting on $X$ was proved in \cite[Thm 2]{Arinkin-Fedorov} in the case $f: X\to B$ has integral fibers. In the influential paper \cite[Prop. 2.1]{Markushevich},  constructs $A \to B^{nm}$  as an algebraic space  acting on the open subset $X^{nc}$, and then shows that it is algebraic over an open subset containing codimension one points. Below, we give a more detailed proof of this result and show that the action on $X^{nc}$ extends to all of $X$ and, furthermore, that  $(X, A)$ is a $\delta$-regular weak abelian fibration. The existence of a smooth relative group scheme acting on a Lagrangian fibration that admits local sections has also been obtained recently by Kim \cite{Kim-Neron}, with completely different methods.
Viewed only as a sheaf of commutative groups, the Albanese fibration also appears in \cite[\S 2.2]{Abasheva-Rogov}. 
%Building on \cite{Markman-TS}, this paper sets the foundation for the study of Shafarevich-Tate groups and this sheaf of commuatative groups is the fundamental character that allows the definition of Shafarevich-Tate twists of a Lagrangian fibration.
\end{rem}

\begin{rem}
At this time, we have not pursued the extension of $A$ over the whole $B$.
%the case when $B^{nm} \subsetneq B$, i.e., when $f$ does not admit a local section over every point of $B$.
\end{rem}

\begin{defin} \label{defin-almost}
Let $ G \to S$ be a smooth group scheme of finite type over a scheme $S$ and let $X \to S$ a  quasi-projective scheme over $S$ with an action of $G$ over $S$. The $G$-scheme $X$ is called an \emph{almost} G--torsor if for any \'etale multisection $s: V \to X$ of $X  \to S$, the open subscheme $X^*_V\subset X_V$ parametrizing fibers meeting the $V$-section $s_V: V \to X_V$ is a (trivial) $G$-torsor.
\end{defin}

Thus, up to \'etale base change, $X$ is covered by open subsets which are trivial $G$-torsors. The following is a simple example of an almost-torsor that is not a torsor.

\begin{example}
Let $f: S \to \P^1$ be an elliptic K3 surface, and let $G \to \P^1$ be the identity (irreducible) component of the relative automorphism group. Then $G \to \P^1$ is a smooth commutative group with connected fibers acting on $S$. The open subset $S^\circ \subset S$ of regular points of $f$ is an almost $G$-torsor, but is not a $G$-torsor as soon as $f$ has reducible fibers because then $S^\circ \to \P^1$ has disconnected fibers.
\end{example}

We start with the following intermediate result. This is essentially  \cite[Proposition 2.1]{Markushevich}. We give a complete proof because it will be needed in the proof of Theorem \ref{thm alb}.

\begin{prop} \label{prop arnold-liouville}
Let $f: X \to B$ be a Lagrangian fibration, let $V \to B^{nm}$ be the map induced by an \'etale multisection, and let $X^{nc,*}_V \subseteq X_V$ be the open subset of the base change parametrizing fibers of $X^{nc}_V$ meeting the corresponding section $s$ of $X_V \to V$. Then  $X^{nc,*}_V \to V$ is a smooth commutative algebraic group scheme with connected fibers and with zero section given by $s$.

%%Then locally in the \'etale topology $X^{nc} \to B^{nm}$ is a smooth commutative group scheme. More precisely, the choice of an \'etale section $s: V \to X^{nc}$ determines a group structure on the base change $X_V^{nc}$.
%%
%%In particular, for every $b \in B^{nm}$, all connected components of $X_b^{nc}$ are isomorphic to each other and they are all isomorphic to an extension of an abelian variety $P_b$ by a commutative affine group $R_b$ and moreover, $P_b$ and $R_b$ depend only on $b \in B$ and not on the connected component.
\end{prop}
\begin{proof}
This is essentially the Arnold-Liouville theorem in the analytic category.
To simplify notation, up to restricting $B$ we can assume that  $B=B^{nm}$. 

Since $f$ is a Lagrangian fibration, there is an isomorphism $ f^* \Omega_B \cong T_{X|B} $ (see \cite[Lemma 2.3.1]{dCRS} for more details\footnote{If $B$ is not smooth, the proof of \cite[Lemma 2.3.1]{dCRS} is not correct. Since we have already observed that $B^{nm}$ is smooth, we do not have to worry about this.}.)
In particular, up to restricting $B$, we can assume that $\Omega_B$ is free, and hence that the subsheaf $f^* \Omega_B \cong T_{X/B}  \subset T_{X}$ is free.
Let $X_1, \dots, X_n$ be global vector fields trivializing $f^* \Omega_B $. 
%Since $f$ is a Lagrangian fibration, these vector fields pairwise commute, i.e. $[X_i,X_j]=0$, for every $i, j$. 
%Moreover, they are tangent to the fibers of $f$.
We claim that these vector fields induce a holomorphic action $a: \mathbb C^n \times X \to X$ (which is the same as a relative action of $ \mathbb C^n \times_B X$ on $X$ over $B$). Indeed, since the vector fields are tangent to the fibers of $f$, which are compact because $f$ is proper, for each $i$ the local flow extends to a global flow defining an action  of $\mathbb C$ on $X$ (cf. Korollar pg 81 Kaup \cite{Kaup}). Since $f$ is Lagrangian, the vector fields pairwise commute,  i.e. $[X_i,X_j]=0$, for every $i, j$ (cf. \cite[Proposition 1]{Beauville-intsystem}. Thus, the action of the individual vector fields extends to a holomorphic action of $\mathbb C^n$ on $X$.

To simplify the notation, up to passing to an \'etale base change, we can assume that $X \to B$ has a section $s$. 
Let
\[
o_s: \mathbb C^b \times B \to X^{nc}
\]
be the orbit morphism of the section  $s: B \to X^{nc}$. 
On the open subset $X^{nc}$ where the rank of $f$ is maximal, the $X_i$ span the relative tangent bundle $ T_{X^{nc}|B} $, which is a subbundle  of $T_{X^{nc}}$. Thus, the orbit morphism $o_s$ is smooth of relative dimension $0$. In particular, the relative stabilizer subgroup $\Stab (s(B)) \to B$ is a discrete subgroup of the trivial group scheme $\mathbb C^n \times B$. 
We claim that $\Stab (s(B)) \subset \mathbb C^n \times B$ does not depend on the chosen section. Indeed, let  $s_1$ and $s_2$ be two sections of $X^{nc} \to B$, and consider the two closed subgroups  $\Stab (s_1(B)), \Stab (s_2(B)) \subset \mathbb C^n \times B$. 
Since the morphism $\Stab (s_1(B)), \Stab (s_2(B)) \to B$ are smooth and hence admit local sections, it suffices to show that the two subgroups agree over an open subset of $B$, but this is clear since over an open subset the fibers of $f$ are abelian varieties. 
The statement about  the connected components of the fibers of  $X^{nc} \to B$ now follows easily, since every connected component  is an orbit and since all open orbits in the same fiber have the same stabiliser.

Set $\Lambda:=\Stab(s(B))$, then $(\mathbb C^n \times B) \slash \Lambda$ is a smooth commutative \emph{analytic} group over $B$ and the morphism $o_S$  determines a biholomorphism $(\mathbb C^n \times B) \slash \Lambda \cong X^{nc,*}$. Here, $X^{nc,*} \subset X^{nc}$ is the union of all the connected component of the fibers that meet the image of the section. 
This subset is open because $X^{nc} \to B$ is flat.

We claim that $X^{nc,*} \to B$ is a smooth commutative \emph{algebraic} group scheme over $B$. We only need to prove that the group laws determined by the isomorphism with $(\mathbb C^n \times B) \slash \Lambda$ are algebraic. We follow an idea of Markushevich:
 By GAGA, the group structure of $X^{nc,*}$ over $B$  is algebraic over the locus where $X^{nc,*} \to B$ is proper, hence the holomorphic mappings defining the group structure are rational algebraic maps. Since analytically they extend, they extend algebraically, and thus
 $X^{nc,*} \to B$ is an algebraic group over $B$.

 \end{proof}

 \begin{rem} \label{rem-gamma-ar} 
 In the first version of this article it was claimed without a complete proof that the group structure extended to all of $X^{nc}$ (thanks to A. Abasheva for pointing this out). This statement is wrong in general, as there are examples of Lagrangian fibrations with sections which have the following property: the birational involution $-1$ determined by the group structure on the smooth fibers is not a regular morphism and  its indeterminacy locus intersects $X^{nc}$ non-trivially. Examples of this are given in \cite[\S 3]{ASF}. They are degree-$0$ Beauville-Mukai systems on a linear system $|H|$ on a K3 surface $S$ with reducible members, where the stability is considered with respect to a general polarization. The $-1$-involution is a birational involution that is not biregular. By \cite[Remark 3.12]{IntJac} and \cite[\S 6]{Mongardi-Onorati}, examples of this phenomenon are given also by the intermediate Jacobian fibration associated with a cubic fourfold containing a plane or a cubic scroll.

%%  Note that the proof above also shows that $\Lambda=j_*(R^{2n-1} f_* \Z_{X_{B^o}})$, where $j: B^o \to B$ is the open immersion of the open subset parametrizing smooth fibers. (Thanks to C. Voisin for this observation). Note also that the subgroup $\Lambda$ is precisely the group $\Gamma$ of \cite[ \S 3.1]{Abasheva-Rogov}.

 \end{rem}

 \begin{rem} \label{rem-sing-alb}
 If $f: X \to B$ is a Lagrangian fibration on a singular $\Q$-factorial terminal symplectic variety, then $f(\Sing (X)) \subset B$ has codimension $\ge 2$ \cite[Lemma 3.2]{Matsushita-base}. Hence, if $X \to B$ does not have multiple fibers in codimension $1$, we can define the Albanese fibration of $X \to B$ as the Albanese fibration of the restriction of $X$ to $B \setminus f(\Sing (X))$.
 \end{rem}

 The  circle of ideas used above also shows that the relative Albanese fibration of a Lagrangian fibration $X \to B$ with $B=B^{nm}$ can be characterized as the unique (up to isomorphism) smooth commutative group scheme with connected fibers with a faithful action on $X \to B$ that makes $X^{nc}$ into an almost torsor. We record this in the following proposition.

\begin{prop} \label{prop_albuniversal} Let $X \to B$ be a Lagrangian fibration with $B=B^{nm}$. Let $A' \to B$ be a smooth commutative separated group scheme with connected fibers and suppose there is a faithful action of $A'$ on $X/B$ which makes $X^{nc}$ into an almost $A'$-torsor. Then $A'$ is isomorphic to the relative Albanese fibration $A \to B$ of $X \to B$.
\end{prop}

\begin{proof}
Let $U \subset B$ be the locus parametrizing smooth fibers. Then $A'_U \cong \Aut^0(X_U \slash U) \cong A_U$. In particular, there is a birational map $\phi: A' \dashrightarrow A$ over $B$.
 Let $V \to B$ be an \'etale base change determined by an \'etale section of $X^{nc} \to B$. The trivializations induced by this section give isomorphisms $ A_V \cong X^{nc,*}_V \cong A'_V$. It follows that \'etale locally the birational map $\phi$ extends to an isomorphism and hence that $\phi$ extends to an isomorphism.
\end{proof}

 \begin{proof}[Proof of Theorem \ref{thm alb}]  As above, in order to simplify notation we will assume that $B=B^{nm}$.
 Consider the relative automorphism group  $\Aut(X/B) \to B$. This is an open subset of the relative Hilbert scheme $ \Hilb (X \times_B X)$ and hence it is separated.
 We will define $A \to B$ as a locally closed subgroup scheme of $\Aut(X/B) \to B$.
  Let $\Aut(X/B)^\circ  \subseteq \Aut(X/B)$ be the irreducible component of $\Aut(X/B)$ containing the zero section (note that in general  there is no reason for this to be flat beyond codimension $1$, nor for this to be a subgroup scheme). If $U \subset B$ denotes the locus parametrizing the smooth fibers of $X \to B$, then $\Aut(X/B)^\circ$ is the closure in $\Aut(X/B)$ of the identity component $\Aut^0(X_U/U) \subset \Aut(X_U/U)$. Since for a family of abelian varieties taking the identity component of the relative automorphism group commutes with base change, the formation of $\Aut(X/B)^\circ $ commutes with \'etale base change. In other words, if $V \to B$ is an \'etale base change, then $(\Aut(X/B)^\circ )_V =\Aut(X_V/V)^\circ$.

%  Note that a priori, this is not a subgroup scheme. , since it is the closure of the locus $\Aut^0(X_U/U)$ in  $\Aut(X/B)$and it is the unique component of , the formation of $\Aut(X/B)^\circ $ commutes with \'etale base change, i.e., if $V \to B$ is an \'etale base change, then $(\Aut(X/B)^\circ )_V =\Aut(X_V/V)^\circ$.
  
 Let $V \to B$ be an \'etale base change with the property that $X^{nc}_V \to V$ has a section $s$. We will identify an open subset $A_V \subset \Aut(X/B)^\circ_V$ that is a smooth commutative group scheme over $V$ with connected fibers and then show that this open subset descends to a smooth commutative group scheme  $A \to B$ that satisfies all the desired properties.

In the proof of Proposition \ref{prop arnold-liouville} above we showed that $(\mathbb C^n \times B) \slash \Lambda_V\cong X_V^{nc,s}$ acts on $X_V/V$, where $\Lambda_V $ is a discrete subgroup. Thus it defines a morphism $a: X_V^{nc,s}  \to \Aut(X_V/V)$ of group schemes over $V$.  Here as usual  $X^{nc,s} \subset  X^{nc}$ is the open subset of fiber components meeting the section $s$. Note that since the $X_V^{nc,s} $ is irreducible, the image of $a$ lands in $\Aut(X_V/V)^\circ$.
The composition $b=o_s \circ a: X_V^{nc,s} \to X_V^{nc}$ of $a$ with the orbit map $o_s: \Aut(X/B)^\circ_V \to X^{nc}_V$  is the  natural open immersion (indeed, it is the identity on the open subset where $f$ is proper).  It follows that $a: X_V^{nc,s} \to  o_s^{-1}(X_V^{nc,s})$ and $o_s:o_s^{-1}(X_V^{nc,s}) \to X_V^{nc,s}$ are inverse of each other and thus isomorphisms.  We set $A_V:=o_s^{-1}(X_V^{nc,s}) \cong a(X_V^{nc,s}) \subset \Aut(X/B)^\circ_V$.

We claim that $A_V \subset \Aut(X/B)^\circ_V$ does not depend on the choice of the section. Indeed, suppose $s_1, s_2: V \to X^{nc}_V$ are two sections.
As already noted above and in the proof of Proposition \ref{prop arnold-liouville},  Zariski locally on $V$ the action of $X_V^{nc,s}$ on $X_V$ corresponds to the action of  $(\mathbb C^n \times B) \slash \Lambda_V$ on $X_V$. In particular, if $a_{i}: X_V^{nc,i} \to \Aut(X/B)^\circ_V$ is the morphism associated to the section $s_i$  as in the previous paragraph,  then $a_1=a_2 \circ t_{s_1-s_2}$.
In particular, the open subset $A_V \subset \Aut(X/B)^\circ_V$  does not depend on the chosen section and, if $n: \Aut(X_V/V)^\circ \to \Aut(X_B/B)^\circ$ denotes the \'etale base change map, then $A_V=n^{-1} (n(A_V))$.

We  define $A \subset \Aut(X/B)^\circ$ to be the union of the open subsets $n(A_V)$ over an \'etale covering of $B$ where $X^{nc}$ acquires a section. Then $A$ is a locally closed subset of $ \Aut(X/B)$. Since $A_U=\Aut^0(X_U/U) $, the group laws on $\Aut(X/B)$ restrict to rational group laws on $A$. As in the proof of Proposition \ref{prop arnold-liouville}, it suffices to check \'etale locally that these extend to regular morphisms. Since the base change of $A$ to $V$ is just $A_V$, this is clear. Note that by construction, $X \to B$ is an almost $A$-torsor. (Alternatively, one can show that for an \'etale cover $B'  \to B$ associated to the choice of an \'etale section over every point of $B$, the scheme $A_{B'}$, together with the restriction of a $B'$-ample line bundle that is the pullback on a line bundle on $\Aut(X \slash B)$ admits a descent datum, which is thus effective by \cite[\S 6.2/7]{BLR}.)

We now check that the pair $(X,A)$ is a weak abelian fibration.  The fact that the stabilizers of all points are affine can be shown in the following way. Since the fibers of $X \to B$ are connected, by Lemma 5.16 of \cite{Arinkin-Fedorov} it is enough to check that the stabilizers are affine on one point of each fiber. This follows from the fact that for every $b \in B$,  the action on $X^{nc}_b$ has trivial stabilizers. The condition on the Tate-modules follows from \cite{Ancona-Fratila}. The $\delta$-regularity follows from \cite[Proposition 2.3.2]{dCRS}.
\end{proof}

The next Proposition shows how to find extendable holomorphic symplectic forms on the relative Albanese variety of a Lagrangian fibration (cf. Definition \ref{def-rel-alb}) and, more generally, on torsors over the relative Albanese variety or isogenous group schemes). For related results in this direction, see Propositions 2.3 and 2.6 of \cite{Markushevich}, or \cite[Cor. 3.7]{Abasheva-Rogov}.

\begin{prop}\label{prop_gamma}
Let $X \to B$ be a Lagrangian fibration with holomorphic symplectic form $\sigma_X$ that is extendable. Let $A \to B^{nm}$ be the relative Albanese variety and let $P \to B^{nm}$ be a quasi-projective almost $A$-torsor. Then there exists a holomorphic symplectic form $\sigma_P$ on $P$ that is extendable.
\end{prop}

\begin{proof} The proposition is implied by Proposition \ref{prop_gamma2} below.
%Let $\wt P \to B$ be a smooth partial compactification of $P$ with a projective morphism to $B$.
% Let $\nu: V \to U$ a surjective \'etale morphism onto an open subset of $B$ corresponding to the choice of an \'etale section of $f$. Up to restricting $U$ we can assume that $\nu$ is proper and that $U$ is contained in with an isomorphism $\varphi_V: X_V \to P_V$. By Zariski's main theorem, up to restricting $U$, we can assume that $\nu$ is proper.
\end{proof}

\begin{prop}\label{prop_gamma2} Let $g: A \to B$ be a quasi-projective smooth commutative group scheme with connected fibers and proper general fiber. Let $X$ be a smooth quasi-projective variety with a smooth morphism $f: X \to B$ with the property that $X^{nc} \to B$ is an almost $A$-torsor. Then there exists an extendable holomorphic symplectic form on $X^{nc}$ with respect to which $f$ is Lagrangian if and only if there exists an extendable holomorphic symplectic form on $A$ with respect to which $g$ is Lagrangian.
%%Let $X^{nc} \to B$ be a Lagrangian fibration with holomorphic symplectic form $\sigma_X$ that is extendable. Let $A \to B^{nm}$ be the relative Albanese variety and let $P \to B^{nm}$ be a quasi-projective almost $A$-torsor. Then there exists a holomorphic symplectic form $\sigma_P$ on $P$ that is extendable.
\end{prop}

Thanks to C. Voisin for suggesting this proof.

\begin{proof}
Let $\wt X$ (respectively $\wt A$) be a smooth partial compactification of $X$ (respectively of $A$) with a projective morphism to $B$ extending $f$ (resp. $g$). Let $\sigma_{X}$ be an extendable holomorphic symplectic form on $X$  and let $\sigma_{\wt X}$ be its extension to $\wt X$.
Let $\nu: V \to B$ be an \'etale morphism associated with an \'etale multisection $s$ of $X \to B$.  Up to restricting $V$, we can assume that $\nu: V \to \nu(V)$ is proper. The action of $A$ on $X$ determines an isomorphism  $\varphi_V: X_V^* \cong A_V$, where as in the proof of Theorem \ref{thm alb}, $X_V^* \subset X_V$ denotes the open subset parametrizing fibers of $X_V \to V$ that meet the tautological section.  
Let $ \wt X_V $ (resp.  $\wt A_V$) be the base change of $\wt X$ (resp. $\wt A$) to $V$, and let $\Gamma_V \subset \wt X_V \times_ V \wt A_V$ be the graph of the birational map $\varphi_V:  \wt X_V \dashrightarrow  \wt A_V$. Let $\Gamma  \subset \wt X \times_B \wt A$ be the closure of $\frac{1}{\deg(\nu)^2} \nu_* \Gamma_V $.
By Proposition \ref{prop_cor_ext}, the holomorphic $2$-form $\sigma_{\wt A}:=\Gamma_* (\sigma_{\wt X})$ is an extendable. 

We now check that the restriction $\sigma_A$ of  $\sigma_{\wt A}$ to $A$ is non-degenerate. This can be done analytically locally in the following way. Since $f$ is Lagrangian with respect to $\sigma_X$, the fibers of $g$ are isotropic with respect to $\sigma_A$. It follows that $\sigma_A$ induces a section $\sigma_A^{1,1} \in H^0(g^*\Omega_B \otimes \Omega_{A/B})$, and $\sigma_A$ is non-degenerate if and only if  $\sigma_A^{1,1}$ induces an isomorphism $g^*T_B \to \Omega_{A/B}$ (cf. Proposition \ref{prop_nondeg}).
%Note that both $\Omega_B $ and $ f_*\Omega_{X/B}$ are rank $n$ vector bundles on $B$. 

Let $U \subset B$ be an analytic open subset over which $f$ is proper and $s$ breaks into a disjoint union of sections $s_0, \dots, s_N: U \to X_U$. Each section determines an isomorphism $\phi_i: A_U \to X_U^i$. By construction, $\sigma_{A_U}=\frac{1}{N} \sum \phi_i^* \sigma_{X_U}=\frac{1}{N}  \phi_0^*\sum t_i^* \sigma_{X_U^i} $, where $t_i=\phi_i \circ \phi_0^{-1}: X_U \to X_U$ is the translation by $s_0-s_i$.  Similarly for $\sigma_{A_U}^{1,1}$.
Under the isomorphism $t_i^*\Omega_{X_U/U} \cong \Omega_{X_U/U}$, $t_i^* \sigma_{X_U}^{1,1}=\sigma_{X_U}^{1,1}$, because translations act trivially on the holomorphic $1$-forms of the general fiber. Hence
\[
\sigma_{A_U}^{1,1}=\phi_0^* \sigma_{X_U^0}^{1,1}=\phi_i^* \sigma_{X_U^0}^{1,1}
\]
does not depend on the section. This shows that $\sigma_A$ is non-degenerate over the locus $W$ of $\nu(V)$ where the fibers are proper.
 We cover $B$ by analytic open subsets $U_\alpha$ over which there exists an analytic sections $s_\alpha: U_\alpha \to X_{U_\alpha}$. Each determines an isomorphism $\phi_\alpha: A_{U_\alpha}  \to X_{U_\alpha}^*$. Here, as usual $ X_{U_\alpha}^* \subset  X_{U_\alpha}$ is the open subset parametrizing fibers meeting the section $s_\alpha$. By pullback, we get holomorphic symplectic forms $\sigma_\alpha=\phi_\alpha^* \sigma_A$ on $A_{U_\alpha}$.  The argument above shows that on $ {A}_{W \cap U_\alpha}$, $\sigma_{A}^{1,1}$ is equal to $\sigma_{A_{U_\alpha}}^{1,1}$ as sections of $g^*\Omega_B \otimes \Omega_{A/B}$, hence $\sigma_{A}^{1,1}$ equals to  $\sigma_{\alpha}^{1,1}$ on all of $U_\alpha$ and hence $\sigma_A$ is non-degenerate on $A_{U_\alpha}$.

The argument to show that if $A$ has an extendable holomorphic symplectic form then so does $X$ is completely analogous.
\end{proof}

%%\section{Topological Obstructions}
%%
%%DO I HAVE ANYTHING TO SAY? CAN YOU EXTEND BROSNAN's OBS to the Q-factorial case? If I am dealing with moduli spaces of sheaves do I know that the local structure of the singularities is such that $IC=\mathbb Q$? Can I prove that the local structure of the singularities is preserved under flops and under locally trivial deformations? Would this give me that if $IC=\mathbb Q$ for the moduli space, then the same is true for all deformations and birational guys?

\section{Applications II: Torsors and isogenous fibrations} \label{section_torsors}

The goal of this section is to prove the following theorem

\begin{thm} \label{thm_torsors}
Let $X \to B$ be a Lagrangian fibration.  Let $A \to B^{nm}$ be the relative Albanese fibration and suppose that $B^{nm}$ contains all codimension one points. For any $A' \to B^{nm}$ smooth commutative group scheme with connected fibers that is isogenous to $A \to B^{nm}$ and any quasi-projective almost $A'$-torsor $P \to B^{nm}$ there exists a $\Q$-factorial terminal symplectic compactification $Y$ of $P$.
\end{thm}

\begin{rem}
In the case of $A$-torsors, a completely different point of view is developed in \cite{Abasheva-Rogov} where Shafarevich-Tate twists of Lagrangian fibrations are systematically studied. Here, the twists are constructed quite easily as complex manifolds and the challenge is to show when they are K\"ahler or projective manifolds. We refer the reader to \cite{Abasheva-Rogov} and \cite{Abasheva} for more details.
\end{rem}

Recall the definition of isogeny of commutative group schemes over a base \cite[Definition 4 \S 7.3]{BLR}

\begin{defin}
A homomorphism $f: G  \to  H $ of commutative group schemes of finite type over a base $S$ is called an isogeny if for each $s \in S$ the homomorphism $f_s: G_s  \to  H_s $ is an isogeny, i.e., $f_s$ is finite and surjective on the identity components. 
\end{defin}

\begin{cor}\label{cor-ratsect}
Let $X \to B$ be a Lagrangian fibration with reduced fibers in codimension one.  Then there exists a $\Q$-factorial terminal symplectic  projective compactification $\bar A $ of $A$. In particular, the Lagrangian fibration $ \bar A \to B$ has a rational section. If $B=B^{nm}$, then the section is regular.
\end{cor}
\begin{proof}
Let $A \to B^{nm}$ be the relative Albanese fibration. By assumption,  $B^{nm}$ contains all codimension one points. Applying the Theorem to $A \to B^{nm}$ yields  $\bar A \to B$. The zero section of the group scheme $A \to B^{nm}$ extends to a rational section.
\end{proof}

As we will see in Theorem \ref{thm_smoothtorsors}, if $X \to B$ has integral fibers, then $\bar A$ is actually smooth and the section is regular.

\begin{cor} \label{cor-pic}
Let $X \to B$ be a Lagrangian fibration and let $V \subset B^{nm}$ be an open subset over which the relative Picard functor $\Pic_{X_V/V}$ is representable by a separated scheme $\Pic(X_V \slash V)$, locally of finite presentation (e.g. if the fibers of $f$ over $V$ are reduced and irreducible). Let $P \subset \Pic(X_V \slash V)$ be a connected component dominating $V$. If the codimension of the complement of $V$ in $B$ has codimension $\ge 2$, then there exist $\bar P  \to B$, a $\Q$-factorial terminal symplectic compactification of $P$.
\end{cor}
\begin{proof} First note that since $X_{B^{nm}} \to B^{nm}$ is a Lagrangian fibration over a smooth base, by \cite{Matsushita-higher} the higher direct images $R^i f_* \mc O_{X_{B^{nm}}} $ are locally free, commute with base change, and are isomorphic to $\Omega^i_{B^{nm}}$. Hence, $f$ is cohomologically flat in dimension $0$ and by \cite[\S 8.4/1]{BLR}, $ Lie( \Pic(X_{B^{nm}})) \cong R^1 f_* \mc O_{X_{B^{nm}}} $. As a consequence, the algebraic space $\Pic(X_{B^{nm}}  \slash B^{nm})$ is smooth of relative dimension $\dim X/2$.
Since $A \to B^{nm}$ acts on $X_{B^{nm}}$, the choice of a line bundle $L$ on $X_{B^{nm}} \to B^{nm}$ determines a functorial morphism $g_L: A \to  \Pic(X_{B^{nm}} \slash B^{nm})$  over $B^{nm}$ (cf. discussion in 6.3 and 8.1/4  of \cite{BLR}). The differential of $g_L$ along the zero section induces a morphism of vector bundles
\[
dg_L: Lie(A/{B^{nm}}) \to Lie( \Pic(X_{B^{nm}}))\cong R^1 f_* \mc O_{X_{B^{nm}}}
\] 
Now since $A \to {B^{nm}}$ is a Lagrangian fibration, $Lie(A/{B^{nm}}) \cong \Omega^1_{B^{nm}}$ (cf. \cite[\S 2.3]{dCRS}), and hence $Lie(A/{B^{nm}})\cong R^1 f_* \mc O_{X_{B^{nm}}}$. Thus $dg_L$ is an endomorphism of the vector bundle $R^1 f_* \mc O_{X_{B^{nm}}}$.
Now since $\Pic(X_V \slash V)$ is smooth and separated over $V$, the connected component of the identity $\Pic^0(X_V \slash V)$ is a smooth group scheme of finite type over $V$ with connected fibers \cite[15.6.5]{EGAIV3}. Moreover, $g_L(A_V) \subset  \Pic^0(X_V \slash V)$ because $A_V \to V$ is smooth with connected fibers.
By Theorem \ref{thm_torsors} and Proposition \ref{prop-isogeny} below, it is enough to show that  $ g_L: A_V \to \Pic^0(X_V \slash V)$ is an isogeny, since any other component $P$ is a torsor over $\Pic^0(X_V \slash V)$. We first show that $g_L$ is a homomorphism of groups over $V$: indeed, by the theorem of the square this is true over the locus where $X_V \to V$ is smooth; since by assumption $\Pic^0(X_V \slash V)$ is separated, it follows that $g_L$ is a group homomorphism over $V$. By \cite[\S 7.3/1]{BLR}, it is enough to show that the differential $dg_L$ is an isomorphism at every point of $V$. But this follows from the fact that  (a) by codimension reasons $g_L$ extends to an automorphism of the vector bundle $R^1 f_* \mc O_{X}$; (b)$d g_L$ is generically an isomorphism since it is so on the locus where $X_V \to V$ is smooth. Hence, $g_L$ is an isomorphism. 
 \end{proof}

\begin{rem} \label{remPicKim}
 In \cite{Kim-dual}, Kim obtains a compactification of the relative $\Pic^0$ of a Lagrangian fibration $X \to B$, where $X$ is one of the known \hk manifolds (i.e. $X$ is of K3$^{[n]}$, generalized Kummer, OG$6$, or $OG10$-type). This compactification is obtained as a quotient of $X$ by a finite group acting symplectically on $X$. The difference here is that in Theorem \ref{thm_torsors}, and hence a priori also in the isogeny of Corollary \ref{cor-pic}, we allow kernels of isogenies that are quasi-finite over the base, but not constant.
%%\begin{enumerate}
%%\item The same conclusion holds if instead of considering the relative Picard functor, one considers a smooth component of a moduli space of (stable complexes of) sheaves on $X$, that has a dense open subset parametrizing line bundles on the smooth fibers of $X \to B$.
%%\item In \cite{Kim-dual}, Kim obtains a compactification of the relative $\Pic^0$ of a Lagrangian fibration $X \to B$, where $X$ is one of the known \hk manifolds (i.e. $X$ is of K3$^{[n]}$, generalized Kummer, OG$6$, or $OG10$-type). This compactification is obtained as a quotient of $X$ by a finite group acting symplectically on $X$. The difference here is that in Theorem \ref{thm_torsors}, and hence a priori also in the isogeny of Corollary \ref{cor-pic}, we allow kernels of isogenies that are quasi-finite over the base, but not constant.
%%\end{enumerate}
\end{rem}

Before proving the theorem we recall the following examples:

\begin{example} \label{ex-og10}  (see \cite{Rapagnetta10, dCRS})
Let $(S,H)$ be a general degree two K3 surface and consider the two Mukai vectors $v:=(0,2H,k)$ and $w:=(0,2H,j)$, where $k$ is an even integer and $j$ is an odd integer. The moduli space $M_{v,H}$ of $H$-semistable sheaves with Mukai vector $v$ is a singular symplectic variety and it admits a symplectic resolution $ \wt M_{v,H} \to M_{v,H}$ which is a  projective \hk manifold of O'Grady 10 type. The moduli space $M_{w,H}$ of $H$-semistable sheaves with Mukai vector $w$ is projective \hk manifold of K3$^{[5]}$-type. The support morphism induces both \hk manifolds with a Lagrangian fibration to $\P^5=|2H|$. The  relative Picard variety $P:=\Pic^0_H(\mc C/|2H|)$ of $H$-stable line bundles of degree $0$ is a smooth commutative group scheme with connected fibers, with an action on $ \wt M_{v,H}$ and $ M_{w,H} $ with affine stabilizers and which turns the smooth locus of the Lagrangian fibrations into almost $P$-torsors (see \cite{dCRS}). Note that by the universal property of the Albanese fibration (Proposition \ref{prop_albuniversal}), $P_{B^{nm}}$ is isomorphic to the Albanese fibration of $N$ and $M$.
\end{example}

The irreducible component of the fibers of $ \wt M_{v,H} \to |2H|$ and $  M_{w,H} \to |2H|$ are computed in \cite[Prop. 4.4.3]{dCRS}. Here it is shown that the two fibrations have fibers with a different number of components over the $2$-dimensional locus in $|2H|$ parametrizing non-reduced curves. 
This example, namely, the example of Lagrangian fibrations over the same base  that are locally isomorphic over an open subset, but have fibers over the same points with a different number of components, indicates that the relation between Shafarevich-Tate twists and compactification of different torsors over the Albanese fibration should be investigated further.

\begin{example} \label{ex-og6} Building on \cite{Rapagnetta-topological}, in \cite{MRS} it is shown that the  Lagrangian fibered OG$6$-type manifold associated with a general principally polarized abelian surface is the target of a relative degree $2$ rational isogeny from a genus $3$ Beauville-Mukai system associated with the corresponding Kummer K3. In \cite{Floccari-Kum6}, Floccari showed that a certain Debarre system on a general principally polarized abelian surface, which is a Lagrangian fibered \hk $6$-fold of generalized Kummer type,  admits a relative degree $2^5$ rational isogeny onto a Beauville-Mukai system associated with  the corresponding Kummer K3.
\end{example}

The natural question is whether there are other examples to which the theorem above applies. The two isogenies in Example \ref{ex-og6} are induced by a global action by a finite group acting by birational maps. It would be interesting to have  examples of isogenies not induced by a finite group acting globally, i.e. examples of isogenies whose kernels are not constant group schemes (cf. Remark \ref{remPicKim}).
Before proving the Theorem we record the following proposition.

\begin{prop} \label{prop-isogeny}
Let $G, H \to S$ be  smooth commutative group schemes of finite type with connected fibers over an irreducible base $S$ defined over a field of characteristic $0$. Suppose that there exists an isogeny $\phi: G \to H$ over $B$. Then $\phi$ is \'etale and there exists an $N \in \mathbb N$ and an isogeny $\psi: H \to G$ over $B$ such that $\psi \circ \phi: G \to H$ is multiplication by $N$.
%$\ell:=\deg \phi$.
%%Let $G, H \to S$ be separated smooth commutative group schemes of finite type with connected fibers over an irreducible base $S$ with field of fractions $K$. Suppose that the general fiber is proper and that there exists an isogeny $\phi: G \to H$ over $B$. 
%%Then, up to passing to an open subset $U \subset B$ containing all codimension 1 points, there exists an isogeny $\psi: H_{U} \to G_{U}$ such that $\psi \circ \phi: G \to G$ is multiplication by $\ell:=\deg \phi_{K}$.
\end{prop}
\begin{proof}
For every $s \in S$, $f_s: G_s \to H_s$ is \'etale by \cite[7.3/1]{BLR}.  By \cite[2.4/2]{BLR}, this implies that $f$ is flat, and by \cite[2.4/7]{BLR} that $f$ is \'etale. The closed subgroup $\ker (f) \subset G$ is quasi-finite over $B$. Factorizing $f$ as an open immersion followed by a finite morphism, it follows that there exists an $N$ such that $\ker (f) \subset \ker [N]$, where $[N]: G \to G$ is multiplication by $N$. It follows that there exists a $\psi: H \to G$ such that  $[N]=\psi \circ \phi$. It can be easily checked that $\psi$ is an isogeny (in particular it is \'etale). 
%By Proposition 6 of \cite[\S 7.3]{BLR}, and the fact that general fibers are abelian varieties, there exists an isogeny $\psi_: H_{K(S)} \to G_{K(S)}$ of the generic fibers such that  $\psi_K \circ \phi_K$ is multiplication by $\deg \phi_{K}$. By ARGUMENT NEEDED, up to restricting to an open subset $U \subset B$ containing all codimension 1 points,  this extends to a global morphism $\psi: H \to G$ such that  $\psi \circ \phi: G \to G$ is multiplication by $\ell:=\deg \phi_{K}$. THIS IMPLIES CHECK that $\psi$ is an isogeny.
\end{proof}

%Note that if the generic fiber of $G$ is proper, then it suffices to take $N$ equal to the degree of $f$ on the generic fiber. DO YOU NEED THAT GENERIC FIBER IS PROPER?

\begin{rem}
The proposition shows that in the theorem above, it does not matter if we assume the existence of an isogeny from $A$ to $A'$ or in the opposite direction.
\end{rem}

\begin{rem} If $G$ (or equivalently $H$)  has semistable reduction over all codimension one points of $B$, then using \cite[7.3/6]{BLR}, it can be seen that up to restricting $B^{nm}$ to an open subset containing all codimension one points, we only need to assume the existence of an isogeny $\phi_K: A_K \to A_K'$ (or $\psi_K:A'_K \to A_K$  between the generic fibers.
%The proof of the proposition shows that in the theorem above, up to restricting $B^{nm}$ to an open subset containing all codimension one points, we only need to assume the existence of an isogeny $\phi_K: A_K \to A_K'$ between the generic fibers.
\end{rem}

\begin{proof} By Proposition \ref{prop-isogeny}, there exist isogenies  $\phi: A \to A'$ and $\psi: A' \to A$ over $B^{nm}$.
By Proposition \ref{prop_gamma} (3) there exists an extendable holomorphic symplectic form on $A$. 
%Applying Theorem \ref{thm1}, we get a $\Q$-factorial terminal symplectic projective compactification $\bar A $ of $A$, with a regular morphism $\bar A \to $.
The form  $\psi^*(\sigma_{A})$ is an extendable holomorphic symplectic form on $A'$, which is non-degenerate because $\psi$ is \'etale. 
%is an extendable holomorphic symplectic form, and thus extends to a holomoprhic $2$-form $\sigma_{A'}$, which is non-degenerate for codimension reasons. 
Applying  Proposition \ref{prop_gamma} (3) to $A' \to B^{nm}$ and $P \to B^{nm}$, we get an extendable holomorphic symplectic on $P$. We are thus in the position to apply Theorem \ref{thm1}, which yields a $\Q$-factorial terminal symplectic compactification $Y$ of $P$, with a Lagrangian fibration $Y \to B$, thus concluding the proof.
\end{proof}

Note that in general, we cannot expect  $Y$ to be smooth, as the following examples shows. However, if the fibers are integral we can indeed check smoothness of $Y$, at least in the case when $P$ is an $A$-torsor, as we prove in the Theorem below.

\begin{example} Let $(S,H)$ be a polarized K3 surface of degree $H^2=2g-2$ and let $m \ge 2$ be such that $(m, H^2)\neq (2,2)$. Let $v=(0, mH, -m^2\frac{H^2}{2})$ and $w=( 0, mH,\chi)$ with $gcd(m,\chi)=1$. Let $M=M_{v,H}$, respectively $N=M_{w,H}$, be the moduli spaces of $H$-semistable sheaves on $S$ with Mukai vector $v$, respectively $w$. Then there are Lagrangian fibrations $M \to |mH|$ and $N \to |mH|$, the first of which has a (rational) section. On the open subset parametrizing integral curves, both fibrations are torsors over the relative $\Pic^0$ of the family of curves (in fact, $M$ is a compactification of this abelian group scheme).
 The moduli space $N$ is a smooth compact \hk manifold of K3$^{[n]}$-type and by \cite{Kaledin-Lehn-Sorger}, $M$ is a $\Q$-factorial terminal symplectic variety that is singular and does not admit a symplectic resolution.
In particular, by Theorem \ref{thm1}, no holomorphic symplectic compactification of the Albanese fibration of $N$ is smooth.
\end{example}

\begin{thm} \label{thm_smoothtorsors}
Let $f: X \to B$ be a Lagrangian fibered smooth projective \hk manifold with integral fibers. Let $A \to B$ be the relative Albanese fibration from Theorem \ref{thm alb}. Then for any quasi-projective $A$-torsor $P \to B$, there exists a smooth projective \hk manifold $Y$ compactifying $P$ and a Lagrangian fibration $Y \to B$ extending $P \to B$. Moreover, locally in the \'etale topology, the two fibrations $X \to B$ and $g: Y \to B$ are isomorphic. In particular, $A$ itself admits a \hk compactification $\bar A \to B$ which has a section.
\end{thm}

\begin{rem}
In \cite[Theorem C]{Abasheva-Rogov}, it is proved that if $f: X \to \mathbb P^n$ has integral fibers, then there is a unique Shafarevich-Tate twist of $X$ with a section provided $X$ satisfies $H^3(X, \Q)=0$, and $H^2(\mathbb P^n, \Gamma)=0$ (here $\Gamma$ is the group from Remark \ref{rem-gamma-ar}). In particular, the fibration $f$ can be deformed to a fibration with a holomorphic section through the Shafarevich-Tate family. It should be noticed that while Theorem \ref{thm_smoothtorsors} does not need the cohomological vanishing, it does not say anything about the deformation class of the fibration with a section. With some other assumptions on the Lagrangian fibration $X$ (maximal holonomy and indivisibility of the general fiber), the paper \cite{Bogomolov-Kamenova-Verbitsky} also shows the existence of a Shafarevich-Tate twist of $X$ with a meromorphic section.
\end{rem}

%\begin{rem}
%In the more general case when $X$ is $\Q$-factorial terminal symplectic, the same proof (together with DOING ALB FOR SINGULAR GUYS) gives the existence of a $\Q$-factorial terminal symplectic $Y$ that is locally isomorphic to $X$, provided that $H^0(U, R^2 f_* \Q_X)=\Q$ is equal to one as in the smooth case.
%\end{rem}

Before proving the Theorem above, we should remark that we expect the fibrations obtained above to be deformation equivalent to $X$. As a starting point, we show they have the same rational cohomology.

\begin{prop} \label{prop-coho}
Let $X$ and $Y$ be as in Theorem \ref{thm_smoothtorsors}, then $H^*(X) \cong H^*(Y)$ as rational Hodge structures.
\end{prop}
\begin{proof}
Since by Theorem \ref{thm alb} the pairs $(X,A)$ and $(Y,A)$ are $\delta$-regular weak abelian fibrations with integral fibers, by \cite[7.2.1]{Ngo-lemme}, the decomposition theorem applied to the Lagrangian fibrations is full support. This means that, letting $U \subset \mathbb P^n$ be the locus parametrizing smooth fibers, the perverse sheaves appearing in the decomposition of $Rf_*\mathbb Q_X$, respectively $Rg_*\mathbb Q_Y$, are (shifts of) the intersection complexes $\mc{IC}_{\mathbb P^n} ( R^i f_* \mathbb Q_{X_U})$, respectively  $\mc{IC}_{\mathbb P^n} ( R^i g_* \mathbb Q_{Y_U})$. Since the local systems $ R^i f_* \mathbb Q_{X_U} $ and $  R^i g_* \mathbb Q_{Y_U}$ are naturally isomorphic, the result follows.
\end{proof}

\begin{proof}[Proof of Theorem \ref{thm_smoothtorsors}] By Theorem  \ref{thm_torsors}, there exists a $\Q$-factorial terminal symplectic projective compactification $Y $ of $P$, with a Lagrangian fibration $g: Y \to B$ extending $P \to B$. 
To prove the theorem, it is enough to prove the second statement, namely that $Y \to B$ is locally isomorphic to $X \to B$. Let $ V \to B$ be an \'etale morphism trivializing the $A$-torsors $P \to B$ and $X^{nc} \to B$. Assume $V$ is connected and denote by $X_V \to V$ and $Y_V \to V$ the base change morphisms. The trivialization of the two $A$-torsors induces a birational map $\phi_V: X_V \dashrightarrow Y_V$, which extends the isomorphism $X^{nc}_V \to P_V$. We claim that $\phi_V$ is an isomorphism in codimension $2$.   Let $t: Z \to Y_V$ be a resolution of $Y_V$ with a morphism $s: Z \to X_V$. To start, note that the complement of $X^{nc}_V$ in $X_V$ has codimension $\ge 2$. Thus $\phi_V$ contracts no divisors and hence the exceptional locus $Ex(t)$ of $t$ is contained in the exceptional locus $Ex(s)$ of $s$.
The canonical classes of $X_V$ and $Y_V$ are trivial,  $X_V$ has terminal singularities  (it is smooth), and $Y_V$ has canonical singularities (it  is a symplectic variety). Thus, the negativity Lemma (see pg. 420 of \cite{Kawamata-flops}) shows that $Ex(s) \subset Ex(t)$.  Hence $Ex(s)=Ex(t)$ and $\phi_V$ is an isomorphism in codimension $2$.

By a Matsusaka-Mumford type argument (see  \cite[Thm 13.23]{Hacon-Kovacs}), in order to show that $\phi_V$ is an isomorphism, it suffices to show that there is a relatively ample line bundle on $Y_V$ whose pullback to $X_V$ is relatively ample over $X_V$. This is the content of Proposition \ref{ampleness} below, which ends the proof of Theorem \ref{thm_smoothtorsors}.
\end{proof}

\begin{prop} \label{ampleness}
There exists a relatively ample line bundle $H_{Y_V}$ on $Y_V/V$ with the property that $\varphi_V^*(H_{P_V})$ is relatively ample on $X_V/V$.
\end{prop}
\begin{proof}
Let $t \in B$ be a general point.
% in the image of $V \to B$.
By results of Voisin \cite{Voisin-Lagrangian} and Matsushita \cite{Matsushita-deformation}, it follows that the rank of the restriction morphism $H^2(X,\Q) \to H^2(X_t,\Q)$ is one. Let $U \subset B$ be the locus parametrizing smooth fibers of $X \to B$. The invariant cycle theorem gives
\be \label{restrX}
\im[H^2(X) \to H^2(X_t)]=\im[H^2(X_U) \to H^2(X_t)]=H^0(U, R^2 f_* \Q_X)=\Q
\ee
Let $\nu: \wt Y \to Y$ be a resolution of the singularities of $Y$. Note that $Y$ has rational singularities (it is a symplectic variety) and is $\Q$-factorial. By Lemma \ref{lem_rational} below, $\im[H^2(\wt Y) \to H^2(Y_U)]=\im[H^2(Y) \to H^2(Y_U)]$ and hence, by the invariant cycle theorem,
\be \label{imP}
\im[H^2(Y) \to H^2(Y_t)]=\im[H^2(\wt Y) \to H^2(Y_t)]=H^0(U, R^2 g_* \Q_{Y_U})=\Q.
\ee
By (\ref{restrX}) and (\ref{imP}), there exists ample line bundles $H_X$ on $X$ and $H_Y$ on $Y$ such that under the natural isomorphism $H^2(X_t, \Q) = H^2(Y_t, \Q)$ determined by the choice of a point,
\[
[(H_X)_{|X_t}]=[(H_Y)_{|Y_t}].
\]
Let $H_{Y_V}$ (resp. $H_{X_V}$) be the pullbacks of $H_X$ (resp.  $H_Y$) to $X_V$ (resp.  $Y_V$). These are relatively ample over $V$.
Now consider the pullback $\phi^*H_{Y_V}$ to $X_V$. By construction and the fact that $\phi_V$ is an isomorphism over the smooth locus, it follows that the classes of $(\phi^*H_{Y_V})_{X_t}$ and of $(H_{X_V})_{|X_t}$ coincide for every $t$ for which $X_t$ is smooth. The morphism $X_V \to V$ is flat, projective, with intgeral fibers. Hence the relative Picard functor is represented by a separated scheme. The locus parametrizing numerically trivial line bundles is represented by a closed subscheme (\cite[Thm 4 \S 8.4]{BLR}). Thus, if $\phi^*H_{P_V}$ and $H_{X_V}$ are numerically equivalent on the general fiber, they are numerically equivalent for every fiber (see also \cite[Lemma 4.4]{Voisin-twisted}).
%By the separatedness of the identity component of the relative Picard variety of $X_V \to V$ (which is a projective flat morphism with integral fibers), it follows that the restrictions of $\phi^*H_{P_V}$ and $H_{X_V}$ to any fiber are numerically equivalent (see also \cite[Lemma 4.4]{Voisin-twisted}). 
Since $H_{X_V}$ is ample over $V$, by the Nakai-Moishezon criterion it follows that $\phi^*H_{P_V}$ is also ample over $V$.
\end{proof}

\begin{lemma} \label{lem_rational}
Let $Z$ be a normal variety with $\Q$-factorial and rational singularities and let $v: \wt Z \to Z$ be a resolution of singularities. Then for every $W \subseteq Z^{reg}$ open subset,
\[
\im[H^2( Z) \to H^2(W)]=\im[H^2(\wt Z) \to H^2(W)]=W_0(H^2(W)).
\]
\end{lemma}
\begin{proof} We only need to prove the first equality.
Let  $E_i \subset \wt Z$ be the exceptional divisors of $v: \wt Z \to Z$. By \cite[Prop. 12.1.6]{Kollar-Mori-flips},
\be \label{kollarmori}
\im [H^2(\wt Z) \to H^0(Z, R^2 v_* \Q_{\wt Z})]=\im [\sum \Q [E_i] \to H^0(Z, R^2 \nu_* \Q_{\wt Z})],
\ee
where $[E_i]$ denotes the class of $E_i$ in $H^2(\wt Z)$. Since $Z$ has rational singularities, the Leray spectral sequence for $v$ yields
\be \label{leray}
0 \to H^2(Z, \Q) \to H^2(\wt Z, \Q) \to H^0(Z, R^2 v_* \Q_{\wt Z}).
\ee
Since $[E_i]_{|Z^{reg}}=0$, (\ref{kollarmori}) and (\ref{leray}) give $\im[H^2(\wt Z) \to H^2(Z^{reg})]=\im[H^2( Z) \to H^2(Z^{reg})]$. The statement of the Lemma with an arbitrary $W$ follows, since restriction to open subsets induces a surjection on the lowest piece of the weight filtration.
\end{proof}.

%%Then
%%\be \label{restrP}
%%\im[H^2(\wt Y) \to H^2(Y_t)]=\im[H^2(Y_U) \to H^2(Y_t)]=H^0(U, R^2 { p}_* \Q_{P})=H^0(U, R^2 f_* \Q_X)=\Q.
%%\ee
%%Let  $E_i \subset \wt P$ be the exceptional divisors of $\nu: \wt Y \to Y$.  Since $Y$ is a symplectic variety it has rational singularities (see \cite{Beauville-sympl-sing}) and  by construction it is $\Q$-factorial. Hence, by \cite[Prop. 12.1.6]{Kollar-Mori-flips},
%%\be \label{kollarmori}
%%\im [H^2(\wt Y) \to H^0(Y, R^2 \nu_* \Q_{\wt Y})]=\im [\sum \Q [E_i] \to H^0(Y, R^2 \nu_* \Q_{\wt Y})],
%%\ee
%%where $[E_i]$ denotes the class of $E_i$ in $H^2(\wt Y)$.
%%The fact that $Y$ has rational singularities also implies that the Leray spectral sequence for $\nu$ yields
%%\be \label{leray}
%%0 \to H^2(Y, \Q) \to H^2(\wt Y, \Q) \to H^0(Y, R^2 \nu_* \Q_{\wt Y}).
%%\ee
%%Since $[E_i]_{|Y_t}=0$ in $H^2(Y_t, \Q)$, using (\ref{kollarmori}) and (\ref{leray}), gives
%%\be \label{imP}
%%\im[H^2(Y) \to H^2(Y_t)]=\im[H^2(\wt Y) \to H^2(Y_t)]
%%\ee

\bibliographystyle{alpha}

\bibliography{bibliographyPicAlb.bib}

\end{document}